\documentclass[12pt]{amsart}

\usepackage{hyperref}
\hypersetup{nesting=true,debug=true,naturalnames=true}
\usepackage{graphicx,amssymb,upref}

\usepackage{amsmath,amsthm}
\usepackage{amsfonts}
\usepackage{latexsym}
\usepackage{amsbsy}
\usepackage{amssymb,mathscinet}
\usepackage{tikz-cd}
\usepackage{mathtools}

\numberwithin{equation}{section}

\usepackage{amsfonts,amssymb,amscd,amsmath,enumerate,verbatim,calc,enumerate}
\theoremstyle{plain}
\newtheorem{theorem}{Theorem}[section]
\newtheorem{corollary}[theorem]{Corollary}
\newtheorem{lemma}[theorem]{Lemma}

\newtheorem{notation}[theorem]{Notation}
\newtheorem{observation}[theorem]{Observation}

\newtheorem{definition}[theorem]{Definition}

\newtheorem{remark}[theorem]{Remark}

\newcommand{\be}{\mathbb E}

\newcommand{\ot}{\otimes}
\newcommand {\id} {{\textrm{id}}}
\newcommand{\wt}{\widetilde}
\newcommand{\wT}{\wt{T}}

\newcommand{\wV}{\wt{V}}

\begin{document}
	
	\title[A Characterization of invariant subspaces in the product system] 
	{A characterization of invariant subspaces for isometric representations of product system over $\mathbb{N}_0^{k}$}

	\date{\today}
	\author[Saini]{Dimple Saini \textsuperscript{*}}
	\address{Centre for mathematical and Financial Computing, Department of Mathematics, The LNM Institute of Information Technology, Rupa ki Nangal, Post-Sumel, Via-Jamdoli Jaipur-302031, (Rajasthan) India}
	\email{18pmt006@lnmiit.ac.in, dimple92.saini@gmail.com}
	\author[Trivedi]{Harsh Trivedi}
	\address{Centre for mathematical and Financial Computing, Department of Mathematics, The LNM Institute of Information Technology, Rupa ki Nangal, Post-Sumel, Via-Jamdoli Jaipur-302031, (Rajasthan) India}
	\email{trivediharsh26@gmail.com, harsh.trivedi@lnmiit.ac.in}
	\author[Veerabathiran]{Shankar Veerabathiran}
	\address{Institute of Mathematical Sciences,
		A CI of Homi Bhabha National Institute,
		CIT Campus, Taramani, Chennai, 600113,
		Tamilnadu, INDIA.}
	\email{shankarunom@gmail.com}

	\begin{abstract}
		Using the Wold-von Neumann decomposition for the isometric covariant representations due to Muhly and Solel, we prove an explicit  representation of the commutant of a doubly commuting pure isometric representation of the product system  over $\mathbb{N}_0^{k}.$ 
		As an application we study a complete characterization of invariant subspaces for a doubly commuting pure isometric representation of the product system. This provides us a complete set of isomorphic invariants. Finally, we classify a large class of an isometric covariant representations of the product system. 
	\end{abstract}

	
	
	\maketitle
	
\section{Introduction}
Beurling's theorem \cite{AB49}, is a classical application of the Wold-von Neumann decomposition  \cite{N29,W38}, says that: Let $\mathcal K$ be a closed subspace of the Hardy space $H^2(\mathbb D)$ over unit disc $\mathbb D,$  and let $\mathcal{K}$ be an invariant subspace for multiplication operator $M_z$ then $\mathcal K=\theta H^2(\mathbb D)$ for some inner function $\theta$ in $H^{\infty}(\mathbb D).$
In Multivariable Operator Theory, the corresponding question is that whether or not we can prove Beurling-type characterization for all invariant subspaces of the shift $(M_{z_1},\dots, M_{z_n})$ on the  Hardy space over polydisc $H^2(\mathbb D^n).$ It has been first explained by Rudin \cite{W69} where it is pointed out that the Beurling theorem does not hold in this case. For $n>1,$ the parallel problem of the study of invariant subspaces of $H^2(\mathbb D^n)$ is therefore interesting and there is pioneering work by  Agrawal-Clark-Douglas \cite{ACD86}, Ahern-Clark \cite{A70}, Guo \cite{G99},  Izuchi \cite{I87} and by several other authors mentioned in the references of these papers.

Using the doubly commuting condition to characterize the existence of the wandering subspace for commuting tuple of isometries, S{\l}oci{\'n}ski described a Wold-type decomposition for doubly commuting isometries in \cite{Sl80}. Mandrekar \cite{M88} used the S{\l}oci{\'n}ski's Wold-type decomposition and discussed a Beurling-type characterization for $H^2(\mathbb D^2)$ under the condition that the tuple $(M_{z_1}, M_{z_2})$ is doubly commuting.
Based on a generalization of the S{\l}oci{\'n}ski's decomposition due to Sarkar \cite{JS14},
Sarkar-Sasane-Wick in \cite{SSW13} worked out the Mandrekar's result mentioned above for the polydisc case $H^2(\mathbb D^n)$.
Berger-Coburn-Lebow decomposition \cite{BCL78} was revisited by Maji-Sarkar-Sankar in \cite{MSS18} and the explicit decomposition in this paper is based on a representation theorem for commutators of shifts. On the other hand using the explicit representation theorem for commutators, a large class of $n$-tuples of commuting isometries is classified by Maji-Mundayadan-Sarkar-Sankar in \cite{MMSS19} where the following Beurling-Lax-Halmos type theorem is provided:
\begin{theorem}\label{th1}
	Let $\mathcal{E}$ be a Hilbert space, $\mathcal{K}$ be a closed subspace of $H^2_{\mathcal{E}_n}(\mathbb D),$ where $\mathcal{E}_n=H^2(\mathbb D^n)\ot \mathcal{E}$ and let $\mathcal{W}=\mathcal{K}\ominus z\mathcal{K}.$ Then $\mathcal{K}$ is invariant for $(M_z,M_{\kappa_1},\dots,M_{\kappa_n})$ if and only if there exist an inner operator $\Theta\in H^{\infty}_{B(\mathcal{W},\mathcal{E}_n)}(\mathbb D)$ and $(M_{\Phi_1} ,\dots,M_{\Phi_n})$ is an $n$-tuple of commuting shifts on $H^2_{\mathcal{W}}(\mathbb D)$ such that $$\mathcal{K}=\Theta H^2_{\mathcal{W}}(\mathbb D)\quad and \quad \kappa_i\Theta=\Theta \Phi_i,$$ where $\Phi_i(w)=P_{\mathcal{W}}(I_{\mathcal{K}}-wP_{\mathcal{K}}M_z^*)^{-1}M_{\kappa_i}|_{\mathcal{W}}$ for $i\in I_n$ and  $w\in \mathbb D.$
\end{theorem}
\noindent Along this direction, Maji-Sankar studied mixed invariant doubly commuting subspaces in the polydisc case in \cite{MS21}.

Cuntz \cite{C77} studied a $C^*$-algebra, known as Cuntz algebra, generated by row isometries, that is, the tuple $(V_1,\dots,V_n)$ of isometries with orthogonal range. Frazho \cite{F84} proposed the Wold decomposition for two isometries with the orthogonal range. Popescu \cite{P89} extended this decomposition to the case of an infinite sequence of isometries with orthogonal final spaces and extended the Beurling-Lax-Halmos theorem in \cite{P89b}, where Beurling-Lax-Halmos type multiplier/representer $\Theta$ in the noncommutative setup is discussed and it is defined on a Fock space. In this case, the map $\Theta$ intertwines corresponding shifts on the respective Fock spaces and it is known as a {\it multi-analytic} operator. The monograph by Ball-Bolotnikov \cite{BB22} discussed transfer-function realization-type formulas for the Beurling-Lax-Halmos representer using Noncommutative Function-theoretic Operator Theory. Popescu considered Operator Theory on noncommutative domains and studied weighted  Beurling-Lax-Halmos type representer in monograph \cite{P10}. Recently Ball-Bolotnikov \cite{BB22b} studied the representer $\Theta$ for free noncommutative multivariable setting. In \cite{DPS21}, Das-Pradhan-Sarkar developed the following theorem for doubly noncommuting shifts which is a generalization of Theorem \ref{th1} using the Fock space approach:

\begin{theorem}\label{DPSS3}
	Assume ${\bf S}=({\bf S}_{n_1},\dots, {\bf S}_{n_k})$ to be the $k$-tuple of tuples of noncommuting pure isometries, where ${\bf S}_{n_i}=({\bf S}_{i1},\dots, {\bf S}_{in_i})$ such that ${\bf S}_{ij}{\bf S}_{pq}={\bf S}_{pq}{\bf S}_{ij}$ and ${\bf S}_{ij}^*{\bf S}_{pq}={\bf S}_{pq}{\bf S}_{ij}^*$ for all $1\le i<p\le k, j=1,\dots,n_i$ and $q=1,\dots,n_p.$  Let $\mathcal{M}$ be a closed subspace of the Fock ${\bf n}$-module $\mathcal{F}^2_{\bf n}$ for ${\bf{n}}=(n_1,\dots,n_k) \in \mathbb N^{k}_0,$ $\mathcal{E}=\mathcal{M}\ominus \sum_{j=1}^{n_1}{\bf S}_{1j}\mathcal{M}$ and $\mathcal{E}_{\bf n}=\mathcal{F}^2_{n_2}\ot \dots \mathcal{F}^2_{n_k}.$ Then $\mathcal{M}$ is submodule of $\mathcal{F}^2_{\bf n}$ if and only if there
	exist multi-analytic operators $\Phi_{ij}\in R^{\infty}_{n_1} \overline{\ot} B(\mathcal{E}),$ $j=1,\dots,n_i$ and $i=2,\dots,k,$ and an inner multi-analytic operator $\Theta\in R^{\infty}_{n_1}\overline{\ot} B(\mathcal{E},\mathcal{E}_{\bf n})$ such that 
	$$\mathcal{M}=\Theta (\mathcal{F}^2_{n_1}\ot \mathcal{E}) \quad and \quad {\bf S}_{ij}\Theta=\Theta \Phi_{ij},$$ where the Fourier coefficients of $\Phi_{ij}$ are given by $\phi_{ij,\alpha^t}=P_{\mathcal{E}}(S^{\alpha^*}\ot I_{\mathcal{E}_{\bf n}}){\bf S}_{ij}|_{\mathcal{E}}.$
\end{theorem}

Operator theory and its inseparable counterpart, function theory, are well-established disciplines with a long history of overwhelming connections with other subjects. The overarching
field, in keeping with traditional questions, is undergoing continuous development while also
witnessing the generalizations of classical concepts in various settings. One of the generalized
ideas of classical operator and function theory is the notion of Hilbert $C^*$-modules in terms
of $C^*$-correspondences and covariant representation. The paper is along this
line.

Pimsner in \cite{P97} extended the construction of Cuntz-Kreiger algebras using isometric covariant representations of $C^*$-correspondences and the construction is in terms of Fock modules. The setup of isometric  covariant representations provides a general framework to study row isometries using the tensor product notion of $C^*$-correspondences. Muhly-Solel derived an analogue of Popescu's Wold decomposition for a row isometry known as the Wold decomposition for an isometric covariant representation in \cite{MS99} which is based on Fock module approach. In \cite{TV21} the Beurling-Lax-Halmos type theorem in the case of the covariant representations of a $C^*$-correspondence is discussed where the Beurling-Lax-Halmos type multiplier $\Theta$ is defined on the Fock module as a multi-analytic operator.

In \cite{A89}, Arveson investigated the idea of a tensor product system of Hilbert spaces to categorize $E_0$-semigroups. Fowler \cite{F02} introduced the notion of Fock modules.  Solel \cite{S08} studied the doubly commuting notion of covariant representations of product system of $C^*$-correspondences and discussed the theory of regular dilations. The setting of doubly commuting isometric covariant representation is a generalization of $\Lambda$-doubly commuting row isometries considered by Popescu \cite{PoQ} (see also \cite{JP} for a related doubly noncommuting setup). For the doubly commuting isometric representations, Skalski-Zacharias generalized S{\l}oci{\'n}ski's Wold-type decomposition. Recently, Berger-Coburn-Lebow representation for pure isometric covariant representations of product system over $\mathbb{N}_0^2$ was studied in \cite{STV22} and this explicit representation is based on the joint multi-analytic representation theorem for commutators of the induced representation of a product system. As an application of this joint multi-analytic description, the aim of this study is to find a Beurling-Lax-Halmos type theorem for pure doubly commuting isometric representations of a product system and to characterize the invariant subspaces as in Theorems \ref{th1} and \ref{DPSS3}.

This paper concerns the issue of the delicate structure of the
invariant subspaces of the Hardy space over the polydisc, but in the setting of correspondences. Here we followed the theme of \cite{MMSS19} and obtained analgous results in
the setting of correspondences. The plan of the paper is as follows: In Section 2, a fundamental commutant theorem for a doubly commuting pure isometric representation of the product system is  proved. Section 3, presents a complete characterization of invariant subspaces for a doubly commuting pure isometric representation of a product system, as outlined in \cite{MMSS19}. In  Section 4, an application to nested invariant subspaces of Fock module and a complete set of isomorphic invariants of the corresponding invariant subspaces is proved.

\subsection{Preliminaries for Beurling-Lax-Halmos type theorem}
In this paper, $E$ always a {\it $C^*$-correspondence} over a $C^*$-algebra $\mathcal{B}$ with the left module action  given by a non zero $*$-homomorphism $\phi:\mathcal B\to \mathcal L(E)$  in the following sense: $ b\xi :=\phi(b)\xi, \:b\in\mathcal B,\xi\in E,$ 
and $\mathcal L(E)$ denotes the $C^*$-algebra of all adjointable operators on $E.$ Assume all  such $*$-homomorphisms  in this article  are nondegenerate, that is, the closed linear span of $\phi(\mathcal B)E$ equals $E.$  Suppose $F$ is a $C^*$-correspondence over $\mathcal{B},$ then we have the notion of tensor product $F \otimes _{\phi} E$ (cf. \cite{L95}) that  satisfies $$\langle \eta_1 \otimes \xi_1, \eta_2 \otimes \xi_2 \rangle= \langle  \xi_1, \phi(\langle \eta_1, \eta_2\rangle )\xi_2\rangle, \: \eta_{1}, \eta_2 \in F, \xi_1, \xi_2 \in E. $$ By $B(\mathcal{H})$ we denote the algebra of all bounded linear maps acting on the Hilbert space $\mathcal{H}.$  

\begin{definition}
	Let $V:E\to B(\mathcal H)$ be a linear map and $\sigma:\mathcal B\to B(\mathcal H)$ be a representation.
	The pair $(\sigma,V)$ is said to be a {\rm completely bounded covariant representation} (simply say, {\rm c.b.c. representation}) (cf. \cite{MS98}) of $E$ acting on $\mathcal H$  if  $V$ is completely bounded  as a linear map on $E$ (endowed with the usual operator space structure by viewing as a left corner  in  the linking algebra $\mathcal{L}_E$ of $E$) and
	\[
	V(b\xi c)=\sigma(b)V(\xi)\sigma(c) \quad for \quad \xi\in E,
	b,c\in\mathcal B.
	\] Moreover, $(\sigma, V)$ is called {\rm isometric} if $$V(\xi_1)^*V(\xi_2)=\sigma(\langle \xi_1,\xi_2 \rangle)\quad for \quad \quad \xi_1,\xi_2\in E.$$
\end{definition}
Define a bounded linear map $\wV:E\ot \mathcal{H}\to \mathcal{H}$ by $\wV(\xi\ot h)=V(\xi)h,\xi\in E,h\in \mathcal{H},$ then $\wV$ satisfies   $\widetilde{V}(\phi(b)\otimes I_{\mathcal
	H})=\sigma(b)\widetilde{V}$ for all $b\in\mathcal B$. From \cite[Lemma 3.5]{MS98}, $(\sigma, V)$ is a c.b.c. representation ({resp. \it isometric representation}) if and only if $\wV$ is a bounded operator ({resp. isometry}). 
\begin{definition}
	Let $\mathcal{N}$ be a closed subspace of a Hilbert space $\mathcal{H}.$ We say that $\mathcal{N}$ is $(\sigma,V)$-{\rm invariant} (resp. $(\sigma,V)$-{\rm reducing}) (cf. \cite{SZ08}) if it (resp. both $\mathcal{N},\mathcal{N}^{\perp}$) is invariant by every operator $V(\xi)$ for all $\xi \in E$ and $\sigma(\mathcal{B})$-invariant. The restriction  of this representation provides a new representation $(\sigma , V)|_{\mathcal{N}}$ of $E$ on $\mathcal{N}.$	
\end{definition}

\noindent For  $k\in \mathbb{N},$ define $\wV_k : E^{\ot k}\ot \mathcal{H} \to \mathcal{H}$ by $$\wV_k (\xi_1 \ot \dots \ot \xi_k \ot h) = V (\xi_1) \dots V(\xi_k) h,$$ for  $\xi_i \in E, h \in \mathcal H.$
The \emph{Fock space} of   $E$ (cf. \cite{F02}), $\mathcal{F}(E)= \bigoplus_{n \geq 0}E^{\otimes n},$ is a $C^*$-correspondence over $\mathcal{B},$ with the left module action of $\mathcal{B}$ on $\mathcal{F}(E)$ is defined  by $$\phi_{\infty}(b)\left(\oplus_{n \geq 0}\xi_n\right)=\oplus_{n \geq 0}b\xi_n , \:\: \xi_n \in E^{\otimes n}.$$
For $\xi \in E,$  the \emph{creation operator} $V_{\xi}$ on $\mathcal{F}(E)$ is defined by $V_{\xi}(\eta)=\xi \otimes \eta, \:\: \eta \in E^{\otimes n}, n\ge 0.$ Let $\pi $ be a representation of $\mathcal{B}$ acting on a Hilbert space $\mathcal{K}$. Define an  isometric representation $(\rho, S)$ of $E$ acting on a Hilbert space  $\mathcal{F}(E)\otimes_{\pi}\mathcal{K}$  by
\begin{align*}
	\rho(b)=\phi_{\infty}(b) \otimes I_{\mathcal{K}} \quad and \quad  S(\xi)=V_{\xi}\otimes I_{\mathcal{K}} ,\:\: \xi \in E,b \in \mathcal{B},
\end{align*} then we say that such a representation $(\rho, S)$ is referred to as the   {\it induced representation} (cf. \cite{RM74}) induced  by $\pi$.

Suppose that $(\sigma , V)$ is an isometric representation of $E$ acting on $\mathcal{H}.$ By Wold-von Neumann decomposition theorem for $(\sigma , V),$ due to Muhly and Solel \cite{MS99}, there exists a unitary $\Pi_V:\mathcal{H} (=\mathcal{H}_1{\oplus}\mathcal{H}_2)\to (\mathcal{F}(E)\otimes\mathcal{W})\oplus\mathcal{H}_2$ such that
\begin{align*}
	\Pi_V
	\begin{bmatrix}
		\sigma_1(b) & 0 \\
		0 & \sigma_2(b)
	\end{bmatrix}	= \begin{bmatrix}
		\rho(b) & 0 \\
		0 & \sigma_2(b)
	\end{bmatrix}\Pi_V
\end{align*} and
\begin{align*}
	\Pi_V
	\begin{bmatrix}
		V_1(\xi) & 0 \\
		0 & V_2(\xi)
	\end{bmatrix}	= \begin{bmatrix}
		S(\xi) & 0\\
		0 & V_2(\xi) \\
	\end{bmatrix}\Pi_V,
\end{align*} where $b\in \mathcal{B},\xi\in E,$
$\mathcal{H}_1=\bigoplus_{n\geq 0} \wV_n(E^{\ot n}\ot \mathcal{W}),$ $\mathcal{H}_2= \bigcap_{n \geq 1}\wV_n(E^{\otimes  n }\otimes  \mathcal{H})$ and $\mathcal{W}$ is a wandering subspace for $(\sigma,V)$ (that is, $\mathcal{W}\perp\wV_n(E^{\ot n}\ot \mathcal{W})$ for all $n\in \mathbb{N}$). Moreover,   $(\sigma , V)$ is pure  (that is, SOT-$\lim_{n \rightarrow \infty}\wV_{n}\wV^*_{n}=0 $), then $\mathcal{H}_2$ = $\{0\}$  and we have
$$ \Pi_{V}\sigma(b)=\rho(b)\Pi_{V} \quad and \quad \Pi_V V(\xi)=S(\xi)\Pi_V.$$

\noindent Then $\Pi_{V}$ is called {\it Wold-von Neumann decomposition} for the pure isometric representation $(\sigma, V)$. 

\begin{definition} Let $\mathcal{W}$ be a wandering subspace for $(\sigma , V)$ of $E$ acting on $\mathcal{H}.$ We say that $\mathcal{W}$ is {\rm generating wandering subspace}
	if $$\mathcal{H}= \bigvee_{n\in \mathbb{N}_0}\wV_n(E^{\ot n}\ot \mathcal{W}).$$ 
\end{definition}
The following Beurling-Lax-Halmos type theorem is from \cite{TV21} (also see \cite{MS99}):
\begin{theorem}\label{WWWW8}
	Let $(\sigma,V)$ be a pure isometric representation of $E$ acting on a Hilbert space $\mathcal{H}$ and $\mathcal{M}$ be a closed subspace of $\mathcal{H}.$ Then $\mathcal{M}$ is invariant for $(\sigma,V)$ if and only if there exist a Hilbert space $\mathcal{K},$ an inner operator $\Psi:\mathcal{K}_{\mathcal{F}(E)\ot \mathcal{K}}\to \mathcal{H}$ and a representation $\sigma_1$ of $\mathcal{B}$ on $\mathcal{K}$ such that $$\mathcal{M}=M_{\Psi}({\mathcal{F}(E)\ot \mathcal{K}}),$$ where $\mathcal{K}_{\mathcal{F}(E)\ot \mathcal{K}}$ is the generating wandering subspace for an induced representation $(\rho,S)$ of $E$ induced by $\sigma_1.$
\end{theorem}

\section{A fundamental commutant theorem for a doubly commuting pure isometric representation of the product system}
In this section, first we recall the important notion of joint multi-analytic completely bounded multiplier which we introduced in \cite{STV22}. 

For $k\in \mathbb{N},$ we denote $I_k:=\{1,\dots,k\}$ and $\mathbb{N}_{0}:= \mathbb{N} \cup \{0\}.$ A product system $\mathbb{E}$ over $\mathbb N^{k}_0$ is defined by a family of $C^*$-correspondence $\{E_1,\dots,E_k\}$, along with unitary isomorphisms $t_{j,i}: E_j \ot E_i \to E_i \ot E_j$ ($j>i$). And define $t_{i,i} = \id_{E_i \ot E_i}$,  $t_{j,i} = t_{i,j}^{-1}$ for $j<i.$ Then, we identify the $C^*$-correspondence $\be({\bf{n}}),$ for all
${\bf{n}}=(n_1,\dots,n_k) \in \mathbb N^{k}_0$ with $E_1^{\ot^{ n_1}}\ot \dots \ot E_{k}^{\ot^{n_{k}}}$  (see \cite{F02,MS99,TV21,A98}).
First, we start with the following definitions.

\begin{definition}
	The tuple $(\sigma, V^{(1)},\dots, V^{(k)})$ is called {\rm completely bounded, covariant representation} $($simply say, {\rm c.b.c. representation}$)$ (cf. \cite{MS98}) of $\be$ over $\mathbb N^{k}_0$ acting on a Hilbert space $\mathcal{H}$ if each $(\sigma,V^{(i)}), i\in I_k$ is a c.b.c. representation of $E_i$ on $\mathcal{H}$
	and it satisfy the commutative relation
	\begin{equation} \label{rep} \wV^{(i)} (I_{E_i} \ot \wV^{(j)}) = \wV^{(j)} (I_{E_j} \ot \wV^{(i)}) (t_{i,j} \ot I_{{\mathcal H}}),\quad \quad i,j\in I_{k}.
	\end{equation} Moreover, $(\sigma, V^{(1)},\dots, V^{(k)})$ is {\rm isometric (respectively, fully co-isometric)} if each $(\sigma, V^{(i)}),i\in I_{k}$ is isometric {(respectively, fully co-isometric)}.
\end{definition}

\begin{definition}	
	Let $(\sigma , V^{(1)},\dots,V^{(k)})$ be a c.b.c. representation of ${\mathbb{E}}$ over $\mathbb{N}_0^k$ on $\mathcal{H}.$ Let $\mathcal{K}$ be a closed subspace of $\mathcal{H}.$ We say that $\mathcal{K}$ is {\rm invariant} for $(\sigma , V^{(1)},\dots,V^{(k)})$ if it is $(\sigma,V^{(i)})$-invariant for all $i\in I_k.$
\end{definition}

The {\it Fock module} of $\mathbb{E}$ over $\mathbb{N}^{k}_0,$
$\mathcal{F}(\mathbb{E})=\bigoplus_{{\bf n} \in \mathbb{N}^{k}_0}\mathbb{E}({\bf n}),$ is a $C^*$-correspondence over $\mathcal{B},$ with the left action $\phi_{\infty}$ given by $$\phi_{\infty}(b)(\bigoplus_{{\bf{n}} \in \mathbb{N}_0^{k}} \xi_{\bf{n}}) = \bigoplus_{{\bf{n}} \in \mathbb{N}_0^{k}} b\xi_{\bf{n}}, \: b\in \mathcal{B},\xi_{\bf{n}} \in \mathbb{E}(\bf{n}).$$
For ${\bf{n}}=(n_1,\dots,n_{k}) \in \mathbb{N}_0^{k},$ 
define a map $V_{\bf{n}}: \mathbb{E}({\bf{n}}) \rightarrow B(\mathcal{H})$ by $$V_{\bf{n}}(\xi_{\bf{n}})h=\wV_{\bf{n}}(\xi_{\bf{n}} \ot h)\quad \quad  h \in \mathcal{H},\xi_{\bf{n}} \in \mathbb{E}({\bf{n}}),$$ where $\wV_{\bf{n}}: \mathbb{E}({\bf{n}}) \ot \mathcal{H} \rightarrow \mathcal{H}$ is given by 
$$\wV_{\bf{n}}=\wV^{(1)}_{n_1}(I_{E_1^{\ot n_1}} \ot \wV^{(2)}_{n_2})\dots(I_{E_1^{\ot n_1}\ot\dots\ot E_{k-1}^{\ot n_{k-1}}} \ot \wV^{(k)}_{n_{k}}).$$

Let $\mathcal{K}$ be a Hilbert space and $\pi$ be a representation of $\mathcal{B}$ on $\mathcal{K}.$  For $i\in I_{k},$ define an isometric representation $(\rho, S^{(i)})$ of a $C^*$-correspondence $E_i$ acting on a Hilbert space $\mathcal{F}(\mathbb{E}) \otimes_{\pi} \mathcal{K}$ by \begin{align*}
	\rho(b)=\phi_{\infty}(b) \otimes I_{\mathcal{K}} \quad and \quad
	S^{(i)}(\xi_i)=V_{{\xi}_i} \otimes I_{\mathcal{K}},\:\:\: b \in \mathcal{B}\:, \xi_i \in E_i,
\end{align*} where $V_{\xi_i}$ is the {\it creation operator} determined by $\xi_i$ on $\mathcal{F}(\mathbb{E}).$
We say that $(\rho, S^{(1)}, \dots,S^{(k)})$ is an {\it induced representation} (cf.\cite{RM74}) of $\mathbb{E}$ induced by $\pi.$ 

Two such representations  $(\sigma, V^{(1)},\dots, V^{(k)})$ and  $(\psi, T^{(1)},\dots, T^{(k)})$ of $\be$ acting on the Hilbert spaces $\mathcal H$ and $\mathcal{K},$ respectively, are {\it isomorphic} if there exists a unitary $U:\mathcal H \to
\mathcal{K}$ such that   $\psi(b)U=U\sigma(b)$ and $T^{(i)}(\xi_i)U = U V^{(i)} (\xi_i)$ for all $\xi_i \in E_i, b\in \mathcal{B},i \in
I_{k}.$

\begin{notation}
	Let $\mathcal{K}$ be a Hilbert space, $\pi$ be a representation of $\mathcal{B}$ on $\mathcal{K}$ and $\mathbb{E}$ be a product system over $\mathbb{N}_0^{k+1},k\ge 1$. Let $\Theta: E_{k+1} \to B(\mathcal{K}, \mathcal{F}({\mathbb{E}}_{k}) \otimes_{\pi} \mathcal{K})$ be a completely bounded  bi-module map, that is, $\Theta$ is completely bounded  and $\Theta(a \xi b)=\rho(a)\Theta(\xi)\rho(b),$ where $\xi \in E_{k+1}, a,b \in \mathcal{B}$ and ${\mathbb{E}}_{k}$ is a product system determined by the $C^*$-correspondences $\{E_1,\dots, E_{k}\}.$ Define a bounded linear map $\widetilde{\Theta}: E_{k+1} \otimes \mathcal{K} \longrightarrow \mathcal{F}({\mathbb{E}}_{k}) \otimes_{\pi} \mathcal{K}$  by $\widetilde{\Theta}(\xi \otimes h)=\Theta(\xi)h.$ 
	Since $\mathcal{F}({\mathbb{E}}_{k}) \otimes_{\pi} \mathcal{K}= \bigoplus_{{\bf n} \in \mathbb{N}_0^k}\wt{S}_{\bf n}(\mathbb{E}({\bf n })\otimes \mathcal{K}),$ we define a corresponding completely bounded  map $M_{\Theta} :E_{k+1}  \longrightarrow B(\mathcal{F}({\mathbb{E}}_{k}) \otimes_{\pi} \mathcal{K})$ by
	\begin{align*}
		M_{\Theta}(\xi) (S_{\bf n}(\xi_{\bf n}) h)=\widetilde{S}_{\bf n}(I_{\mathbb{E}({\bf n})} \otimes \widetilde{\Theta})(t_{k+1,{\bf n}} \otimes I_{\mathcal{K}})(\xi \otimes \xi_{\bf n} \otimes h),
	\end{align*}
	where  $\xi \in E_{k+1}, \xi_{\bf n} \in {\mathbb{E}({\bf n})}, h\in \mathcal{K}, {\bf n} \in \mathbb{N}_0^{k},$ $(\rho,S^{(1)},\dots,S^{(k)})$ is an induced representation of ${\mathbb{E}}_{k}$ induced by $\pi,$ $\widetilde{S}_{\bf{n}}=\widetilde{S}^{(1)}_{n_1}(I_{E_1^{\ot n_1}} \ot \widetilde{S}^{(2)}_{n_2})\dots(I_{E_1^{\ot n_1}\ot\dots\ot E_{k-1}^{\ot n_{k-1}}} \ot \widetilde{S}^{(k)}_{n_{k}})$ and $t_{k+1,{\bf n}}: E_{k+1} \ot {\mathbb{E}({\bf n})}  \rightarrow {\mathbb{E}({\bf n})} \ot E_{k+1} $ is an isomorphism which is a composition of the isomorphisms $\{t_{i,j}\: :\: 1 \leq i, j \leq k+1\}.$
	Clearly  $M_{\Theta}(\xi)|_{\mathcal{K}}=\Theta(\xi)$ for each $\xi \in E_{k+1},$  $(\rho,{M}_{\Theta})$ is a c.b.c. representation of $E_{k+1}$ on $\mathcal{F}({\mathbb{E}}_{k}) \otimes_{\pi} \mathcal{K}$ and it satisfies
	\begin{align}\label{SS1}
		M_{\Theta}(\xi)\left(\bigoplus_{{\bf n} \in \mathbb{N}_0^k}\xi_{\bf n} \otimes h_{\bf n}\right)=\sum_{{\bf n} \in \mathbb{N}_0^k}\widetilde{S}_{\bf n}(I_{\mathbb{E}({\bf n})} \otimes \widetilde{\Theta})(t_{k+1,{\bf n}} \otimes I_{\mathcal{K}})(\xi \otimes \xi_{\bf n} \otimes h_{\bf n}),
	\end{align} 
	where $\xi \in E_{k+1}, \xi_{\bf n} \in {\mathbb{E}({\bf n})}, h_{\bf n} \in \mathcal{K}.$ It is easy to see that
	$\widetilde{M}_{\Theta} (I_{E_{k+1}} \otimes \widetilde{S}^{(i)})=\widetilde{S}^{(i)} (I_{{E}_i} \otimes \widetilde{M}_{\Theta})(t_{k+1,i} \otimes I_{\mathcal{F}(\mathbb{E}) \otimes \mathcal{K}})$ for all $i\in I_k.$ It follows that $(\rho, S^{(1)},\dots,S^{(k)},M_{\Theta})$ is a c.b.c. representation of $\mathbb{E}$ on $\mathcal{F}({\mathbb{E}}_{k}) \otimes_{\pi} \mathcal{K}.$ Observe that $(\rho,{M}_{\Theta})$ is isometric if and only if $\widetilde{\Theta}$ is isometry.
\end{notation}

The following important result characterize commutant of a doubly commuting pure isometric representation of product system in terms of a multi-analytic operator which is a generalization of \cite[Lemma 2.2]{STV22}.

\begin{lemma}\label{SS2}
	Assume $\mathbb{E}$ to be a product system over $\mathbb{N}_0^{k+1}$. Let   $(\rho,S^{(1)},\dots,S^{(k)})$ be an induced representation of ${\mathbb{E}}_{k}$ induced by $\pi$ and let $(\rho ,V)$ be a c.b.c. representation of a $C^*$-correspondence $E_{k+1}$ on $\mathcal{F}({\mathbb{E}}_{k}) \otimes_{\pi} \mathcal{K}.$  Then $(\rho, S^{(1)},\dots,S^{(k)}, V)$ is a c.b.c. representation of $\mathbb{E}$ on $\mathcal{F}({\mathbb{E}}_{k}) \otimes_{\pi} \mathcal{K}$ if and only if there exists a completely bounded bi-module map  $\Theta: E_{k+1} \to B(\mathcal{K}, \mathcal{F}({\mathbb{E}}_{k}) \otimes_{\pi} \mathcal{K}) $ such that $$V=M_{\Theta}.$$
\end{lemma}
\begin{proof}
	Let $(\rho, S^{(1)},\dots,S^{(k)}, V)$ be a c.b.c. representation of $\mathbb{E}$ on $\mathcal{F}({\mathbb{E}}_{k}) \otimes_{\pi} \mathcal{K}.$ For $i\in I_k,$ we have
	\begin{align*}
		\widetilde{V} (I_{E_{k+1}} \otimes \widetilde{S}^{(i)})=\widetilde{S}^{(i)} (I_{{E}_i} \otimes \widetilde{V})(t_{k+1,i} \otimes I_{\mathcal{F}({\mathbb{E}}_{k}) \otimes \mathcal{K}}),
	\end{align*} 
	Define a map $\Theta: E_{k+1} \to {B}(\mathcal{K}, \mathcal{F}({\mathbb{E}}_{k}) \otimes_{\pi} \mathcal{K}) $ by $\Theta(\xi)=V(\xi)|_{\mathcal{K}}$ for each $\xi \in E_{k+1}.$ From Equation (\ref{SS1}), for $\xi \in E_{k+1}$, $V(\xi)$ is uniquely determined by $\Theta(\xi).$
	Indeed, since $\mathcal{F}({\mathbb{E}}_{k}) \otimes_{\pi} \mathcal{K}= \bigoplus_{{\bf n} \in \mathbb{N}_0^k}\wt{S}_{\bf n}(\mathbb{E}({\bf n })\otimes \mathcal{K})$ and  for each $\xi \in E_{k+1}, \xi_{\bf n} \in \mathbb{E}({\bf n }),{{\bf n} \in \mathbb{N}_0^k}, h \in \mathcal{K},$ we have
	\begin{align*} 
		V(\xi)S_{\bf n}(\xi_{\bf n})h&=\widetilde{V}(\xi \otimes S_{\bf n}(\xi_{\bf n})h)=\widetilde{S}_{\bf n}(I_{\mathbb{E}({\bf n })} \otimes \widetilde{V})(t_{k+1,{\bf n}} \otimes I_{\mathcal{K}})(\xi \otimes \xi_{\bf n} \otimes h)
		\\&=\widetilde{S}_{\bf n}(I_{\mathbb{E}({\bf n })} \otimes \widetilde{\Theta})(t_{k+1,{\bf n}} \otimes I_{\mathcal{K}})(\xi \otimes \xi_{\bf n} \otimes h)
		=M_{\Theta}(\xi)  S_{\bf n}(\xi_{\bf n}) h.
	\end{align*}
	It follows that $V$ is uniquely determined by $\Theta,$ and hence   $V=M_{{\Theta}}.$ The converse part follows from Equation (\ref{SS1}).
\end{proof}



\begin{definition} 
	A c.b.c. representation $(\sigma, V^{(1)},\dots, V^{(k)})$ of $\mathbb{E}$ on $\mathcal H$ is said to be {\rm doubly commuting } if \begin{equation}
		\wV^{(i)^*} \wV^{(j)} =
		(I_{E_i} \ot \wV^{(j)})  (t_{j,i} \ot I_{\mathcal{H}})  (I_{E_j} \ot \wV^{(i)^*}) \quad \quad i\neq j \quad and\quad i,j\in I_{k}.
	\end{equation} We say that $(\sigma, V^{(1)},\dots, V^{(k)})$ is {\rm pure} if each $(\sigma,V^{(i)})$ is pure for all $i\in I_k.$
\end{definition}


Suppose that $(\sigma,V^{(1)},\dots,V^{(k)})$ is a doubly commuting pure isometric representation of $\mathbb{E}$ on $\mathcal H,$ then by \cite[Corollary 3.4]{TV21} 
\begin{equation}\label{ZZZ5}
	\mathcal{H}=\bigoplus_{{\bf n} \in \mathbb{N}^{k}_0}\wV_{\bf n}(\mathbb{E}({\bf n})\ot \mathcal{W}_{\mathcal{H}}),	
\end{equation}
where $\mathcal{W}_{\mathcal{H}}=\bigcap_{i=1}^{k} \mathcal{W}_i(= ker(\wV^{(i)^*}))$ is a generating wandering subspace for $(\sigma , V^{(1)},\dots,V^{(k)}).$ Define a unitary $\Pi_V:\mathcal{H}\to \mathcal{F}(\mathbb{E})\ot \mathcal{W}_{\mathcal{H}}$ by $\Pi_V(\wV_{\bf n}({\bf \xi_n}\ot h))={\bf \xi_n}\ot h$ for all $\xi_{\bf n}\in \mathbb{E}({{\bf n}}),h\in \mathcal{W}_{\mathcal{H}}, {\bf n}\in \mathbb{N}_0^{k}.$ It is easy to verify that
\begin{equation}\label{SS4}
	\Pi_V\sigma(b) = \rho(b)\Pi_V \quad and \quad  \Pi_V V^{(i)}(\xi_i)=S^{(i)}(\xi_i)\Pi_V,
\end{equation} where $b\in \mathcal{B},\xi_i\in E_i, i\in I_{k}$ and an induced representation $(\rho,S^{(1)}, \dots,S^{(k)})$ of $\mathbb{E}$ induced by $\sigma|_{\mathcal{W}_{\mathcal{H}}}.$ It follows that  $(\rho,S^{(1)}, \dots,S^{(k)})$ and $(\sigma , V^{(1)},\dots,V^{(k)})$ of $\mathbb{E}$ acting on the Hilbert spaces $\mathcal{F}(\mathbb{E})\ot \mathcal{W}_{\mathcal{H}}$ and $\mathcal{H},$ respectively, are isomorphic.

The following characterization of commutants (see \cite{NF70,MSS18,MMSS19}) will play a central role in this article.
\begin{theorem}\label{SS6}
	Let $\mathbb{E}$ be a product system over $\mathbb{N}_0^{k+1}.$ Suppose that $(\sigma , V^{(1)},\dots,V^{(k)})$ is a doubly commuting pure isometric representation of ${\mathbb{E}}_{k}$ over $\mathbb{N}_0^k$ on $\mathcal{H}$ and  $\Pi_{V}$ is the Wold-von Neumann decomposition of $(\sigma , V^{(1)},\dots,V^{(k)})$ with the wandering subspace $\mathcal{W}_{\mathcal{H}}$. Let $(\sigma, V^{(k+1)})$ be a c.b.c. representation of $E_{k+1}$ on $\mathcal{H}.$ Define a c.b.c. representation $(\rho, T)$ of $E_{k+1}$ on $\mathcal{F}({\mathbb{E}}_{k}) \otimes \mathcal{W}_{\mathcal{H}}$ by  $$ T(\xi)=\Pi_{V} V^{(k+1)}(\xi) \Pi_{V}^*,$$ where $ \pi= \sigma|_{\mathcal{W}_{\mathcal{H}}}, \rho$ is defined as above and $ \xi \in E_{k+1}.$
	Then  $(\sigma , V^{(1)},\dots,V^{(k+1)})$  is a c.b.c. representation of $\mathbb{E}$ on $\mathcal{H}$ if and only if there exists a completely bounded bi-module map $\Theta : E_{k+1} \rightarrow {B}(\mathcal{W}_{\mathcal{H}}, \mathcal{F}({\mathbb{E}}_{k}) \otimes_{\pi} \mathcal{W}_{\mathcal{H}})$ such that $$T=M_{\Theta}.$$ Moreover, $\Theta(\xi)=\sum_{{\bf n} \in \mathbb{N}_0^k}\widetilde{S}_{\bf n}(I_{\mathbb{E}({\bf n})} \otimes P_{\mathcal{W}_{\mathcal{H}}})\widetilde{V}_{\bf n}^{*}V^{(k+1)}(\xi)|_{\mathcal{W}_{\mathcal{H}}}, ~\xi \in E_{k+1}$ and  $P_{\mathcal{W}_{\mathcal{H}}}$ is an orthogonal projection of $\mathcal{H}$ onto $\mathcal{W}_{\mathcal{H}}.$
\end{theorem}
\begin{proof}
	Let $(\sigma , V^{(1)},\dots,V^{(k+1)})$ be a c.b.c. representation of $\mathbb{E}$ on $\mathcal{H}.$ Since $\Pi_{V}\widetilde{V}^{(i)}=\widetilde{S}^{(i)}(I_{E_i} \otimes \Pi_{V}),i\in I_k$ we have
	\begin{align*}
		\widetilde{T}(I_{E_{k+1}} \otimes \widetilde{S}^{(i)})&=\Pi_{V}\widetilde{V}^{(k+1)}(I_{E_{k+1}} \otimes \Pi^*_{V}\widetilde{S}^{(i)})
		\\&=\Pi_{V}\widetilde{V}^{(k+1)}(I_{E_{k+1}} \otimes \widetilde{V}^{(i)}(I_{E_i} \otimes \Pi^*_{V}))
		\\&=\Pi_{V}\widetilde{V}^{(i)}(I_{E_i} \otimes \widetilde{V}^{(k+1)})(t_{k+1,i} \otimes I_{\mathcal{H}})( I_{E_{k+1} \otimes E_i}\otimes \Pi^*_{V})
		\\&=\widetilde{S}^{(i)}(I_{E_i} \otimes \Pi_{V})(I_{E_i} \otimes \widetilde{V}^{(k+1)})(t_{k+1,i} \otimes \Pi^*_{V})
		\\&=\widetilde{S}^{(i)}(I_{E_i} \otimes \widetilde{T})(t_{k+1,i} \otimes I_{\mathcal{F}({\mathbb{E}}_{k}) \otimes \mathcal{W}_{\mathcal{H}}}).
	\end{align*}
	This implies that $(\rho, S^{(1)},\dots,S^{(k)}, T)$ is a c.b.c. representation of $\mathbb{E}$ on $\mathcal{F}({\mathbb{E}}_{k}) \otimes \mathcal{W}_{\mathcal{H}}$. By Lemma \ref{SS2}, there exists a completely bounded bi-module map $\Theta: E_{k+1} \rightarrow B(\mathcal{W}_{\mathcal{H}}, \mathcal{F}({\mathbb{E}}_{k}) \otimes \mathcal{W}_{\mathcal{H}})$ such that $T=M_{\Theta}.$  Since each $(\sigma, V^{(i)}),i\in I_k$ is pure and from Equation (\ref{ZZZ5}), we have
	\begin{align*}
		\sum_{{\bf n} \in \mathbb{N}_0^k} \widetilde{V}_{\bf n}(I_{\mathbb{E}({\bf n})} \otimes P_{\mathcal{W}_{\mathcal{H}}})\widetilde{V}_{\bf n}^* &=\sum_{{\bf n} \in \mathbb{N}_0^k} \widetilde{V}_{\bf n}(I_{\mathbb{E}({\bf n})} \otimes (I_{\mathcal{H}}-\widetilde{V} \widetilde{V}^*))\widetilde{V}_{\bf n}^{*} =I_{\mathcal{H}}.
	\end{align*}
	Let $w\in\mathcal{W}_{\mathcal{H}}$, then ${\Pi^*_{V}}(w)=w$ and it follows from the above equality, we get   $$V^{(k+1)}(\xi)w=\sum_{{\bf n} \in \mathbb{N}_0^k} \widetilde{V}_{\bf n}(I_{\mathbb{E}({\bf n})} \otimes P_{\mathcal{W}_{\mathcal{H}}})\widetilde{V}_{\bf n}^* V^{(k+1)}(\xi)w, \:\:  \xi \in  E_{k+1}.$$ Therefore\begin{align*}
		\Pi_{V}V^{(k+1)}(\xi)w&=\Pi_{V}\sum_{{\bf n} \in \mathbb{N}_0^k} \widetilde{V}_{\bf n}(I_{\mathbb{E}({\bf n})} \otimes P_{\mathcal{W}_{\mathcal{H}}})\widetilde{V}_{\bf n}^* V^{(k+1)}(\xi)w\\&=\sum_{{\bf n} \in \mathbb{N}_0^k}\widetilde{S}_{\bf n}(I_{\mathbb{E}({\bf n})} \otimes P_{\mathcal{W}_{\mathcal{H}}})\widetilde{V}_{\bf n}^{*}V^{(k+1)}(\xi)w.
	\end{align*} Hence  for $\xi \in E_{k+1},$  $\Theta(\xi)=\sum_{{\bf n} \in \mathbb{N}_0^k}\widetilde{S}_{\bf n}(I_{\mathbb{E}({\bf n})} \otimes P_{\mathcal{W}_{\mathcal{H}}})\widetilde{V}_{\bf n}^{*}V^{(k+1)}(\xi).$

	Conversely, suppose that $T=M_{\Theta},$ then by Lemma \ref{SS2}, it is easy to verify that $(\sigma , V^{(1)},\dots,\\V^{(k+1)})$  is a c.b.c. representation of $\mathbb{E}$ on $\mathcal{H}.$
\end{proof}

\begin{observation}
	(1). In the setting of Theorem \ref{SS6}, if $(\sigma, V^{(k+1)})$ is an isometric representation of $E_{k+1}$ on $\mathcal{H}.$ 
	Then the corresponding bi-module map $\Theta : E_{k+1} \rightarrow {B}(\mathcal{W}_{\mathcal{H}}, \mathcal{F}({\mathbb{E}}_{k})\otimes_{\pi} \mathcal{W}_{\mathcal{H}})$ is isometric.
	
	(2) If $(\sigma , V^{(1)},\dots,V^{(k+1)})$ is a doubly commuting, then by above and \cite[Proposition 4.6]{HS19},  $$\Theta(\xi)= P_{\mathcal{W}_{\mathcal{H}}}V^{(k+1)}(\xi)|_{\mathcal{W}_{\mathcal{H}}}=V^{(k+1)}(\xi)|_{\mathcal{W}_{\mathcal{H}}}, ~\xi \in E_{k+1}.$$ 
\end{observation}

\section{Invariant subspaces for doubly commuting pure isometric representations}

In this section, we characterize invariant subspaces for doubly commuting pure isometric representation of a product system $\mathbb{E}$ over $\mathbb{N}_0^{k+1},k\ge 1$ on $\mathcal{H}$ $($see Theorem \ref{A6} $).$ Now we recall the following definitions and conclusions are from \cite{TV21}.

\begin{definition}
	Let $(\sigma,V^{(1)},\dots, V^{(k)})$ and $(\psi, T^{(1)},\dots, T^{(k)} )$  be two c.b.c. representations of $\mathbb{E}$ acting on the Hilbert spaces  $\mathcal{H}$ and $\mathcal{K},$  respectively. A bounded linear  map $B: \mathcal{H}\to \mathcal{K}$ is said to be {\rm multi-analytic} if 
	\begin{equation}\label{G1}
		B \sigma(b)h= \psi(b)Bh \quad and \quad	BV^{(i)}(\xi_i)h=T^{(i)}(\xi_i)Bh,
	\end{equation}
	for $ h \in \mathcal{H},\:\xi_i \in E_i, b \in \mathcal{B}$ and $ i\in I_{k}.$
\end{definition}

\begin{remark}
	\begin{itemize}
		\item [(1)]	Suppose that   $(\sigma,V^{(1)},\dots, V^{(k)})$ and $(\psi, T^{(1)},\dots, T^{(k)} )$ are  isometric representations of $\mathbb{E}$ acting on the Hilbert spaces   $\mathcal{H}$ and $\mathcal{K},$ respectively. Further, assume that  $(\sigma,V^{(1)},\dots, V^{(k)})$  is pure and doubly commuting, then
		$$ \mathcal{H}=\bigoplus_{\mathbf{n} \in \mathbb{N}_0^{k}}\wt{V}_{\mathbf{n}}(E(\mathbf{n}) \otimes \mathcal{W}_{\mathcal{H}}),$$ where  $\mathcal{W}_{\mathcal{H}}$ is the generating wandering subspace for  $(\sigma,V^{(1)},\dots,V^{(k)}).$ 
		Let $B: \mathcal{H} \to \mathcal{K}$ be a bounded operator that satisfies the intertwining relation (\ref{G1}), then $B$ is determined uniquely by the operator  $\Psi: \mathcal{W}_{\mathcal{H}} \to \mathcal{K}$ such that $\Psi \sigma(b)h=\psi(b)\Psi h$ for all $b\in \mathcal{B},h\in \mathcal{W}_{\mathcal{H}},$ where $ \Psi:= B|_{\mathcal{W}_{\mathcal{H}}}.$ This happens because for every $h \in \mathcal{W}_{\mathcal{H}}$ and $\xi_{\mathbf{n}} \in E(\mathbf{n}),$ we have $BV_{\mathbf{n}}(\xi_{\mathbf{n}})h=T_{\mathbf{n}}(\xi_{\mathbf{n}})\Psi h,$ where $\mathcal{H}=\bigoplus_{\mathbf{n} \in \mathbb{N}_0^{k}}\wt{V}_{\mathbf{n}}(E(\mathbf{n}) \otimes \mathcal{W}_{\mathcal{H}} ).$
		\item [(2)]	On the other hand, if we start a bounded operator $\Psi: \mathcal{W}_{\mathcal{H}} \to \mathcal{K} ( = \bigoplus_{\mathbf{n} \in \mathbb{N}_0^{k}}\wt{T}_{\mathbf{n}}(E(\mathbf{n}) \otimes \mathcal{W}_{\mathcal{K}} ))$  such that $\Psi \sigma(b)h=\psi(b)\Psi h,  h\in \mathcal{W}_{\mathcal{H}}.$ Since $\mathcal{H}=\bigoplus_{\mathbf{n} \in \mathbb{N}_0^{k}}\wt{V}_{\mathbf{n}}(E(\mathbf{n}) \otimes \mathcal{W}_{\mathcal{H}} ),$ define a bounded operator $M_{\Psi}: \mathcal{H}\to \mathcal{K}$ by $$M_{\Psi}V_{\mathbf{n}}(\xi_{\mathbf{n}})h=T_{\mathbf{n}}(\xi_{\mathbf{n}})\Psi h ,\:\:\: \xi_{\mathbf{n}} \in E(\mathbf{n}), h \in \mathcal{W}_{\mathcal{H}}.$$ Observe that $M_{\Psi}$ is the multi-analytic operator  for the representations  $(\sigma,V^{(1)},\dots, V^{(k)})$ and $(\psi, T^{(1)},\dots, T^{(k)} ).$  Moreover, 
		$$M_{\Psi}\left(\bigoplus_{\mathbf{n} \in \mathbb{N}^{k}_0}h_{\mathbf{n}} \right)=\sum_{\mathbf{n} \in \mathbb{N}^{k}_0} \wt{T}_{\mathbf{n}}(I_{E(\mathbf{n})} \otimes \Psi)\wt{V}_{\mathbf{n}}^*h_n, \hspace{1cm} \: \bigoplus_{\mathbf{n} \in \mathbb{N}^{k}_0}h_{\mathbf{n}} \in \mathcal{H}.$$
		\item [(3)]  An operator $\Psi: \mathcal{W}_{\mathcal{H}} \to \mathcal{K}$ satisfying $\Psi \sigma(b)h=\psi(b)\Psi h$ for all $b\in \mathcal{B},h\in \mathcal{W}_{\mathcal{H}}$ is called {\rm inner} if $M_{\Psi}$ is an isometry.
		\item [(4)] $\Psi$ is inner if and only if $\Psi$ is an isometry and $\Psi (\mathcal{W}_{\mathcal{H}})$ is the wandering subspace for  $(\psi, T^{(1)},\dots, T^{(k)} ),$ that means, $\Psi (\mathcal{W}_{\mathcal{H}})\perp\wT_{\bf n}(\mathbb{E}({\bf {n}}) \ot \Psi (\mathcal{W}_{\mathcal{H}}))$ for all $\bf n\in \mathbb{N}_0^k \setminus \{\bf 0\}.$
	\end{itemize}
\end{remark}

The following theorem for a covariant representation of a product system is motivated from \cite[Theorem 3.1]{MMSS19}. 

\begin{theorem}\label{ZZZ2}
	Let $(\sigma , V^{(1)},\dots,V^{(k)},V^{(k+1)},\dots,V^{(k+m)})$ be a doubly commuting pure isometric representation of $\mathbb{E}$ acting on the Hilbert space $\mathcal{H}.$ Then there exist a Hilbert space $\mathcal{K},$ a  representation of $\mathcal{B}$ on $\mathcal{K}$ and an isometric bi-module maps  $\Theta_j:E_{k+j}\to B(\mathcal{K},\mathcal{F}({\mathbb{E}_{k}})\ot \mathcal{K})$ such that $(\sigma , V^{(1)},\dots,V^{(k)},V^{(k+1)},\dots,V^{(k+m)})$ is isomorphic to a covariant representation $(\rho,S^{(1)},\dots,S^{(k)},M_{\Theta_1},\dots,M_{\Theta_m})$ of ${\mathbb{E}}$ on $\mathcal{F}({\mathbb{E}_{k}})\ot \mathcal{K}.$ 	Moreover,  
	$$\widetilde{\Theta}_i(I_{E_{k+i}}\ot \widetilde{\Theta}_j)=\widetilde{\Theta}_j(I_{E_{k+j}}\ot \widetilde{\Theta}_i)(t_{k+i,k+j}\ot I_{\mathcal{K}}),\quad \quad i,j\in I_m.$$
\end{theorem}
\begin{proof}
	Let $(\sigma , V^{(1)},\dots,V^{(k+m)})$ be a doubly commuting pure isometric representation of $\mathbb{E}$ on $\mathcal{H},$ then by \cite[Corollary 3.4]{TV21} $\mathcal{H}=\bigoplus_{{\bf n} \in \mathbb{N}^{k+m}_0}\wV_{\bf n}(\mathbb{E}({\bf n})\ot \mathcal{W}_{\mathcal{H}}),$ where $\mathcal{W}_{\mathcal{H}}$ is the generating wandering subspace for $(\sigma, V^{(1)},$ $ \dots, V^{(k+m)}).$ For ${\bf n}=(n_1,\dots,n_k,n_{k+1},\dots,n_{k+m}) \in \mathbb{N}^{k+m}_0,$ one can identify $\check{\bf n}_{k}=(n_{k+1},\dots,n_{k+m})\in \mathbb{N}^{m}_0$ with $(0,\dots,0,n_{k+1},\dots,n_{k+m})\in \mathbb{N}^{k+m}_0,$ similarly  ${\bf n}_{k} \in \mathbb{N}^{k}_0$ 
	so that ${\bf n}=({\bf n}_{k},\check{\bf n}_{k}) \in   \mathbb{N}^{k+m}_0.$ 
	Therefore  $$\mathcal{H}= \bigoplus_{{\bf n}_{k}\in \mathbb{N}_0^k}\wV_{{\bf n}_{k}}(\mathbb{E}({\bf n}_{k})\ot \mathcal{K}),$$ where ${\mathcal{K}}=\bigoplus_{{\check{\bf n}_{k}} \in \mathbb{N}^{m}_0}\wV_{{\check{\bf n}_{k}}}({\mathbb{E}}({\check{\bf n}_{ k}})\ot \mathcal{W}_{\mathcal{H}}).$  In fact $\mathcal{K}$ is the generating wandering subspace for $(\sigma,V^{(1)},\dots,V^{(k)}).$ Since $\mathcal{K}$ is  $\sigma(\mathcal{B})$-invariant, let $(\rho,S^{(1)},\dots,S^{(k)})$ be the induced representation of $\mathbb{E}_{k}$ over $\mathbb{N}_0^{k}$ induced by $\sigma|_{\mathcal{K}}.$  Define a unitary  map $U:\mathcal{H}\to \mathcal{F}(\mathbb{E}_{k})\ot \mathcal{K}$ by $$U(\wV_{{\bf n}_{k}}(\xi_{{\bf n}_k}\ot h))=\xi_{{\bf n}_k}\ot h, \quad  \quad \xi_{{\bf n}_k}\in \mathbb{E}({{\bf n}_k}),h\in \mathcal{K}.$$ Then, for $\xi_i\in E_i$  and $i\in I_k,$ we have 
	\begin{align*}
		UV^{(i)}(\xi_i) \wV_{{\bf n}_{k}}(\xi_{{\bf n}_k}\ot h)&=U \wV_{e_i+{{\bf n}_{k}}}(\xi_i\ot \xi_{{\bf n}_k}\ot h)=\xi_i\ot\xi_{{\bf n}_k}\ot h\\&=S^{(i)}(\xi_i)(\xi_{{\bf n}_k}\ot h)=S^{(i)}(\xi_i)U \wV_{{\bf n}_{k}}(\xi_{{\bf n}_k}\ot h),
	\end{align*} where $e_i \in \mathbb{N}_0^{k}$  whose  $i^{th}$ entry is 1 and all other entries zero. Similarly, $U$ satisfies  $U\sigma(b)=\rho(b)U$ for all $b\in \mathcal{B}.$ This shows that $(\sigma,V^{(1)},\dots, V^{(k)})$ is isomorphic to $(\rho,S^{(1)},\dots, S^{(k)}).$ For $j\in I_m,$ note that $\mathcal{K}$ is invariant for $(\sigma, V^{(k+j)})$ and $\mathcal{F}(\mathbb{E}_{k})\ot \mathcal{K}= \bigoplus_{{{\bf n}_{ k}}\in \mathbb{N}_0^k}\wt{S}_{{{\bf n}_{ k}}}(\mathbb{E}({{\bf n}_{k}})\otimes \mathcal{K}),$ we define a completely bounded bi-module maps $W^{(j)}:E_{k+j}\to B(\mathcal{F}(\mathbb{E}_{k})\ot \mathcal{K})$ by $$W^{(j)}(\eta)(S_{{{\bf n}_{k}}}(\xi_{{{\bf n}_{k}}})h)=\widetilde{S}_{{{\bf n}_{ k}}}(I_{\mathbb{E}({{\bf n}_{k}})}\ot \wV^{(k+j)})(t_{k+j,{{{\bf n}_{ k}}}}\ot I_{\mathcal{K}})(\eta\ot \xi_{{{\bf n}_{k}}}\ot h),$$  $\eta \in E_{k+j},\xi_{{{\bf n}_{k}}}\in \mathbb{E}({{\bf n}_{k}}),h\in \mathcal{K}$ and ${{\bf n}_{k}}\in \mathbb{N}_0^k.$ Clearly $(\rho,W^{(j)})$ is an isometric representation of $E_{k+j}$ on $\mathcal{F}(\mathbb{E}_{k})\ot \mathcal{K.}$
	Since $(\sigma , V^{(k+1)},\dots,V^{(k+m)})$ is a doubly commuting pure isometric representation, $(\rho,W^{(1)},\dots,W^{(m)})$ is also a doubly commuting pure isometric representation. Also,  it is easy to verify that $(\rho,S^{(1)},\dots,S^{(k)},W^{(1)},\dots,W^{(m)})$ is a doubly commuting pure isometric representation of $\mathbb{E}$  on $\mathcal{F}(\mathbb{E}_{ k})\ot \mathcal{K}.$  
	For each $\eta\in E_{k+j},\xi_{{{\bf n}_{k}}}\in \mathbb{E}({{\bf n}_{ k}})$ and $h\in \mathcal{K},$ 
	\begin{align*}
		&UV^{(k+j)}(\eta) U^*|_{\mathbb{E}({{\bf n}_{ k}})\ot \mathcal{K}}(\xi_{{{\bf n}_{k}}}\ot h)=UV^{(k+j)}(\eta)\wV_{{\bf n}_{k}}(\xi_{{\bf n}_k}\ot h)\\&=U\wV^{(k+j)}(I_{E_{k+j}}\ot \wV_{{\bf n}_{k}})(\eta \ot \xi_{{{\bf n}_{k}}}\ot h)\\&=U\wV_{{{\bf n}_{ k}}}(I_{\mathbb{E}({{\bf n}_{ k}})}\ot \wV^{(k+j)})(t_{k+j,{{\bf n}_{ k}}}\ot I_{\mathcal{K}})(\eta \ot \xi_{{{\bf n}_{k}}}\ot h)\\&=(I_{\mathbb{E}({{\bf n}_{k}})}\ot \wV^{(k+j)})(t_{k+j,{{\bf n}_{ k}}}\ot I_{\mathcal{K}})(\eta \ot \xi_{{{\bf n}_{k}}}\ot h)={W}^{(j)}(\eta)(\xi_{{{\bf n}_{k}}}\ot h).
	\end{align*} Then $U\wV^{(k+j)}(I_{E_{k+j}}\ot U^*)=\widetilde{W}^{(j)}, j \in I_m$
	and hence the representation $(\rho,S^{(1)},\dots,S^{(k)},W^{(1)},\\\dots,W^{(m)})$ is isomorphic to $(\sigma , V^{(1)},\dots,V^{(k+m)}).$ By using Lemma \ref{SS2}, there exist isometric bi-module maps  $\Theta_j:E_{k+j}\to B(\mathcal{K}, \mathcal{F}(\mathbb{E}_{ k})\ot \mathcal{K})$  such that $$W^{(j)}=M_{\Theta_j},\quad \quad j\in I_{m}.$$
	
	Observe that  the subspace $\mathcal K$ is $(\rho,W^{(1)},\dots,W^{(m)})$-invariant  (infact $\mathcal{K}$ is  reducing), therefore the bi-module maps 	 $\Theta_j$ satisfies 
	$$\widetilde{\Theta}_i(I_{E_{k+i}}\ot \widetilde{\Theta}_j)=\widetilde{\Theta}_j(I_{E_{k+j}}\ot \widetilde{\Theta}_i)(t_{k+i,k+j}\ot I_{\mathcal{K}}), $$ for all $i,j\in I_m.$
\end{proof}

\begin{observation}
	(1) In the setting of Theorem \ref{ZZZ2}, if $(\sigma,V^{(k+1)},\dots,V^{(k+m)})$ is a c.b.c. (resp. isometric, doubly commuting) representation, then the corresponding $(\rho,M_{\Theta_1},\dots,M_{\Theta_m})$ is a c.b.c. (resp. isometric, doubly commuting) representation. 
	
	(2) Let $(\sigma , V^{(1)},\dots,V^{(k+1)})$ be a doubly commuting pure isometric representation of $\mathbb{E}$ on $\mathcal{H}.$ Then there exist a representation of $\mathcal{B}$ acting on a Hilbert space $\mathcal{K}$ and an isometric bi-module maps  $\Theta_i:E_{i+1}\to B(\mathcal{K})$ such that $(\sigma , V^{(1)},\dots,V^{(k+1)})$ is isomorphic to $(\rho,S,M_{\Theta_1},\dots,M_{\Theta_k})$ of ${\mathbb{E}}$ acting on a Hilbert space $\mathcal{F}({E}_1)\ot \mathcal{K},$ where $$\Theta_i(\xi_i)=V^{(i+1)}(\xi_i)|_{\mathcal{K}}\quad for \quad \xi_i\in E_{i+1},i\in I_k.$$ 
\end{observation}

The Beurling-Lax-Halmos theorem plays an important role to study invariant and doubly commuting invariant subspaces. The next theorem is motivated from theorems \ref{th1},\ref{DPSS3} and \ref{WWWW8}, which is a characterization of invariant subspaces for doubly commuting pure isometric representation. Throughout this paper, we will use $\mathcal{F}({E}_1)\ot \mathcal{K}$ defined in the previous observation unless otherwise stated.

\begin{theorem}\label{A6}
	Suppose that $(\sigma , V^{(1)},\dots,V^{(k+1)})$ is a doubly commuting pure isometric representation of $\mathbb{E}$  on $\mathcal{H}.$ Let $\mathcal{M}$ be a closed subspace of $\mathcal{F}({E}_1)\ot \mathcal{K}.$ Then $\mathcal{M}$ is invariant  for $(\rho,S,M_{\Theta_1},\dots,M_{\Theta_k})$ of ${\mathbb{E}}$ on $\mathcal{F}({E}_1)\ot \mathcal{K}$ if and only if there exist a representation of $\mathcal{B}$ acting on a Hilbert space $\mathcal{W},$ a pure isometric representation  $(\varrho',S^{\mathcal{W}},M_{\Phi_1},\dots,M_{\Phi_k})$ of ${\mathbb{E}}$ on  $\mathcal{F}({E}_1)\ot \mathcal{W}$ and an inner operator  $\Psi :\mathcal{W}\to  \mathcal{K}$ $($for the representations $(\varrho',S^{\mathcal{W}},M_{\Phi_1},\dots,M_{\Phi_k})$ and $(\rho,S,M_{\Theta_1},\dots,M_{\Theta_k}))$ such that $$\mathcal{M}=M_{\Psi}(\mathcal{F}({E}_1)\ot \mathcal{W}),$$ where
	$\Phi_i(\xi_i)=\sum_{n \in \mathbb{N}_0}\widetilde{S}_n^{\mathcal{W}}(I_{E_1^{\otimes n}} \otimes P_{\mathcal{W}}P_{\mathcal{M}})\widetilde{S}_{n}^{*}M_{\Theta_i}(\xi_i)|_{\mathcal{W}}, ~\xi_i \in E_{i+1}, i\in I_{k}.$
\end{theorem}
\begin{proof}
	Let $\mathcal{M}$ be a invariant subspace for $(\rho,S,M_{\Theta_1},\dots,M_{\Theta_k}),$ then the  isometric representation  define by  $(\varrho , T,T^{(1)},\dots,T^{(k)}):=(\rho,S,M_{\Theta_1},\dots,M_{\Theta_k})|_{\mathcal{M}}$ of ${\mathbb{E}}$ on  $\mathcal{M}$ is  pure. Since $\mathcal{W}=\mathcal{M}\ominus \widetilde{T}(E_1\ot \mathcal{M})$ is $\varrho$-invariant, in fact $\mathcal{W}$ is a generating wandering subspace for $(\varrho,T),$ consider  the induced representation  $(\varrho',S^{\mathcal{W}})$   of $E_1$ induced by $\varrho|_{\mathcal{W}}.$
	Let $\Pi_T:\mathcal{M}\to \mathcal{F}(E_1)\ot \mathcal{W}$ be the Wold-von Neumann decomposition of $(\varrho,T),$  that is, $\Pi_T$ is unitary and  satisfies
	\begin{equation}\label{WD}
		\Pi_{T}\varrho(b)= \varrho'(b)\Pi_{T}  \quad and \quad	\Pi_{T}\widetilde{T}=\widetilde{S}^{\mathcal{W}}(I_{E_1 }\ot \Pi_{T}), \quad b\in \mathcal{B}.
	\end{equation} For $i \in I_k$ and since $(\varrho,T)$ is pure isometric, apply Theorem \ref{SS6} $(k=1$ case$)$ for the representation $(\varrho,T,T^{(i)}),$ there exist isometric bi-module maps $\Phi_i:E_{i+1}\to B(\mathcal{W},\mathcal{F}({E}_1)\ot \mathcal{W})$ such that
	\begin{equation*}
		\Pi_{T}\widetilde{T}^{(i)}=\widetilde{M}_{\Phi_i}(I_{E_{i+1}}\ot \Pi_{T}),
	\end{equation*}
	where $\Phi_i(\xi_i)=\sum_{n \in \mathbb{N}_0}\widetilde{S}_n^{\mathcal{W}}(I_{E_1^{\otimes n}} \otimes P_{\mathcal{W}})\widetilde{T}_{n}^{*}T^{(i)}(\xi_i)|_{\mathcal{W}}, ~\xi_i \in E_{i+1}.$ Thus,  $(\varrho',M_{\Phi_1},\dots,M_{\Phi_k})$ is a pure isometric representation of a product system over $\mathbb{N}_0^k$ determined by $\{E_2,\dots, E_{k+1}\}$ on $\mathcal{F}({E}_1)\ot \mathcal{W}$ and by Lemma \ref{SS2} $(\varrho',S^{\mathcal{W}},M_{\Phi_1},\dots,M_{\Phi_k})$ is also a pure isometric representation of $\mathbb{E}$ on $\mathcal{F}({E}_1)\ot \mathcal{W}.$ Let $i_{\mathcal{M}}$ be the inclusion map from  ${\mathcal{M}}$ to $\mathcal{F}({E}_1)\ot \mathcal{K}.$ Define an isometry $\Pi_{\mathcal{M}}: \mathcal{F}(E_1)\ot \mathcal{W}\to \mathcal{F}({E}_1)\ot \mathcal{K}$ by 
	\begin{align*}\label{A7}
		\Pi_{\mathcal{M}}=i_{\mathcal{M}}\Pi_{T}^*,
	\end{align*} then  $\Pi_{\mathcal{M}}\Pi_{\mathcal{M}}^*=i_{\mathcal{M}}i_{\mathcal{M}}^*$ and  range of $\Pi_{\mathcal{M}}={\mathcal{M}}.$ Since  $\mathcal{M}$ is  invariant for $(\rho,S,M_{\Theta_1},\dots,M_{\Theta_k})$ and by Equation (\ref{WD}), 
	we have
	\begin{equation}\label{A3}
		\Pi_{\mathcal{M}}\varrho'(b)=\rho(b)\Pi_{\mathcal{M}}, \quad 	\Pi_{\mathcal{M}}\widetilde{S}^{\mathcal{W}}=i_{\mathcal{M}}\wT(I_{E_1}\ot \Pi_{T}^*)=\widetilde{S}(I_{E_1}\ot \Pi_{\mathcal{M}})
	\end{equation} and \begin{equation}\label{A4}
		\Pi_{\mathcal{M}}\widetilde{M}_{\Phi_i}=i_{\mathcal{M}}\wT^{(i)}(I_{E_{i+1}}\ot \Pi_{T}^*)=\widetilde{M}_{\Theta_i}(I_{E_{i+1}}\ot \Pi_{\mathcal{M}}).
	\end{equation} Then $\Pi_{\mathcal{M}}$ is an isometric multi-analytic operator and  by  the discussion of mutli-analytic operator, there exists an inner operator $\Psi:\mathcal{W}\to  \mathcal{K}$ such that  $\Pi_{\mathcal{M}}=M_{\Psi}.$ This shows that   $$\mathcal{M}=\Pi_{\mathcal{M}}(\mathcal{F}(E_1)\ot \mathcal{W})=M_{\Psi}(\mathcal{F}(E_1)\ot \mathcal{W}).$$ The converse part  follows from the definition of  multi-analytic operator.
\end{proof}

\begin{remark}\label{A91}
	Let  $\mathcal{M} \subseteq \mathcal{F}({E}_1)\ot \mathcal{K}$ be an invariant subspace for $(\rho,S,M_{\Theta_1},\dots,M_{\Theta_k})$ of ${\mathbb{E}}.$   In particular $\mathcal{M}$ is invariant  for $(\rho,S),$ by Theorem \ref{WWWW8}, therefore there exist a $\rho$-invariant subspace  $\mathcal{W}$ and an inner operator $\Psi:\mathcal{W}\to \mathcal{K}$ (for the representations $(\varrho',S^{\mathcal{W}})$ and $(\rho,S)$) such that $$\mathcal{M}=M_{\Psi}(\mathcal{F}({E}_1)\ot \mathcal{W}),$$ where  $(\varrho',S^{\mathcal{W}})$ is the induced representation  induced by $\rho|_{\mathcal{W}}.$ In fact $\mathcal{W}=\mathcal{M}\ominus \widetilde{S}(E_1\ot \mathcal{M}).$ Note that, for each $i \in I_k,$ $\wt M_{\Theta_i}({E_{i+1}} \ot \mathcal{M})=\wt M_{\Theta_i}(I_{E_{i+1}} \ot M_{\Psi})(E_{i+1} \otimes \mathcal{F}({E}_1)\ot \mathcal{W} )  \subseteq M_{\Psi}(\mathcal{F}({E}_1)\ot \mathcal{W}).$ Let $\zeta_i \in E_{i+1} \otimes \mathcal{F}({E}_1)\ot \mathcal{W}, \wt M_{\Theta_i}(I_{E_{i+1}} \ot M_{\Psi})(\zeta_i) \in   range (M_{\Psi}) $ and  thus there exists a unique $\eta_i \in \mbox{ker}M_{\Psi}^{\bot}$ such that $\wt M_{\Theta_i}(I_{E_{i+1}} \ot M_{\Psi})(\zeta_i)=M_{\Psi}(\eta_i).$ It follows from the above, for $i \in I_k$ define the linear maps $\widetilde{X}^{(i)}: E_{i+1} \otimes \mathcal{F}({E}_1)\ot \mathcal{W} \rightarrow  \mathcal{F}({E}_1)\ot \mathcal{W} $ by $\widetilde{X}^{(i)}(\zeta_i)=\eta_i.$ Then it is easy to observe that $\widetilde{X}^{(i)}$ is the  bounded module map, that is, $\widetilde{X}^{(i)}(\phi_{i+1}(a) \ot I_{\mathcal{F}({E}_1)\ot \mathcal{W}})=\varrho'(a)\widetilde{X}^{(i)}, a \in \mathcal{B},$ where $\phi_{i+1}$ is a left action on $E_{i+1},$ and it satisfies 
	\begin{equation}\label{A33}
		\wt M_{\Theta_i}(I_{E_{i+1}} \ot M_{\Psi})=M_{\Psi}\widetilde{X}^{(i)}.
	\end{equation}
	Therefore by using \cite[Lemma 3.5]{MS98}, we can define a completely bounded covariant representation  $(\varrho', X^{(i)})$ of $E_{i+1}$ on   $\mathcal{F}({E}_1)\ot \mathcal{W}$ by $X^{(i)}(\xi_i) h=\wt X^{(i)}(\xi_i \ot h),  \xi_i\in E_{i+1},h\in \mathcal{F}({E}_1)\ot \mathcal{W},i\in I_k.$ Moreover, by using Equation (\ref{A33}), one can varify that  $(\varrho',X^{(1)},\dots,X^{(k)})$ is a pure isometric representation of the product system  determined by $\{E_2,\dots, E_{k+1}\}$ that satisfies \begin{align*}
		&\widetilde{X}^{(i)}(I_{E_{i+1}}\ot \widetilde{S}^{\mathcal{W}})=M_{\Psi}^*\widetilde{M}_{\Theta_i}(I_{E_{i+1}}\ot M_{\Psi}\widetilde{S}^{\mathcal{W}})\\&=M_{\Psi}^*\widetilde{M}_{\Theta_i}(I_{E_{i+1}}\ot\widetilde{S}(I_{E_1}\ot M_{\Psi}))=M_{\Psi}^*\widetilde{S}(I_{E_1}\ot \widetilde{M}_{\Theta_i})(t_{i+1,1}\ot M_{\Psi})\\&=\widetilde{S}(I_{E_1}\ot \widetilde{X}^{(i)})(t_{i+1,1}\ot I_{\mathcal{F}({E}_1)\ot \mathcal{W}}).
	\end{align*}
	This shows that $(\varrho', S^{\mathcal{W}},X^{(1)},\dots,X^{(k)})$  is the pure isometric representaion and by Lemma \ref{SS2} (for $k=1$ case), there exist isometric bi-module maps $\psi_i: E_{i+1} \to B(\mathcal{W}, \mathcal{F}(E_1) \otimes \mathcal{W}) $ such that $X^{(i)}=M_{\psi_i}.$  Finally, by comparing  Equations (\ref{A4}) and (\ref{A33}), we obtain $\psi_i=\phi_i, $ for all $i \in I_k.$
\end{remark}

\begin{remark}\label{A92}
	Let $\mathcal{M}$ be a closed invariant subspace for $(\rho,S,M_{\Theta_1},\dots,M_{\Theta_k})$ of ${\mathbb{E}}$ on $\mathcal{F}({E}_1)\ot \mathcal{K}.$ Let $\mathcal{W},\Psi$ and  $\Phi_i:E_{i+1}\to B(\mathcal{W},\mathcal{F}({E}_1)\ot \mathcal{W})$ be as in Theorem \ref{A6}. Since $P_{\mathcal{M}}=M_{\Psi}M_{\Psi}^*,$ for $n \in \mathbb{N}, \wT_n^*=(I_{E_1^{\ot n}}\ot P_{\mathcal{M}})\widetilde{S}_{n}^{*}=(I_{E_1^{\ot n}}\ot M_{\Psi})\widetilde{S}_{n}^{\mathcal{W}^*}M_{\Psi}^*$ so that $\Phi_i$ can be written as $$\Phi_i(\xi_i)=\sum_{n \in \mathbb{N}_0}\widetilde{S}_{n}^{\mathcal{W}}(I_{E_1^{\otimes n}} \otimes P_{\mathcal{W}}M_{\Psi})\widetilde{S}_{n}^{\mathcal{W}^*}M_{\Psi}^*M_{\Theta_i}(\xi_i)|_{\mathcal{W}}, ~\xi_i \in E_{i+1}, i\in I_k.$$
\end{remark}

\begin{remark}\label{A93}
	Let $\mathcal{M}$ be an invariant subspace for a pure isometric representaion $(\sigma, V)$ of $E$ on  $ \mathcal{H},$ then  by  Theorem \ref{WWWW8} there exist a Hilbert space $\mathcal{W}$, a representaion of $\mathcal{B}$ on $\mathcal{W}$  and  an inner operator $\Psi :\mathcal{W}\to \mathcal{H}$ such that $\mathcal{M}=M_{\Psi}(\mathcal{F}({E}_1)\ot \mathcal{W}),$ where $\mathcal{W}=\mathcal{M}\ominus \widetilde{S}(E_1\ot \mathcal{M}).$ Then $(\rho, S)$  is isomorphic to $(\sigma, V)|_{\mathcal{M}}.$ This result builds the stage for an analogous outcome in the following way: 
	Suppose that  $\mathcal{M}$ is  an invariant subspace for $(\rho,S,M_{\Theta_1},\dots,M_{\Theta_k})$ on $\mathcal{F}({E}_1)\ot \mathcal{K},$ by  Theorem \ref{A6}, $M_{\Psi} :\mathcal{F}({E}_1)\ot \mathcal{W}\to \mathcal{M}(=M_{\Psi}(\mathcal{F}({E}_1)\ot \mathcal{W}))$  is unitary. It is easy to observe that  $M_{\Psi}$  which restricted to range of $M_{\Psi}$ is also intertwine the representaions  $(\rho,S,M_{\Theta_1},\dots,M_{\Theta_k})|_{\mathcal{M}}$ and $(\varrho',S^{\mathcal{W}} ,M_{\Phi_1},\dots,M_{\Phi_k}).$ This shows that $(\rho,S,M_{\Theta_1},\dots,M_{\Theta_k})|_{\mathcal{M}}$  and  $(\varrho',S^{\mathcal{W}} ,M_{\Phi_1},\dots,M_{\Phi_k})$ are isomorphic. Let $\mathcal{N}$ be an invariant subspace for a doubly commuting pure isometric  representation $(\sigma , V^{(1)},\dots,V^{(k+1)})$ of $\mathbb{E}$ on  $\mathcal{H}. $ By Theorem \ref{ZZZ2}, there exists a unitary operator $U: \mathcal{H} \to \mathcal{F}({E}_1)\ot \mathcal{K}$  such that $U$ intertwine  the representaions $(\sigma , V^{(1)},\dots,V^{(k+1)})$ and $(\rho,S,M_{\Theta_1},\dots,M_{\Theta_k}).$ Define  $\mathcal{M}'=U(\mathcal{N}),$  then $\mathcal{M}'$ is invariant for $(\rho,S,M_{\Theta_1},\dots,M_{\Theta_k})$ and by the above $(\sigma , V^{(1)},\dots,V^{(k+1)})|_{\mathcal{N}}$ is isomorphic to $(\varrho',S^{\mathcal{W}} ,M_{\Phi_1},\dots,M_{\Phi_k}).$ 
\end{remark}

\begin{remark}(Uniqueness of Theorem \ref{A6}). Let $\pi$ be a representation of $\mathcal{B}$ acting on a Hilbert space $\mathcal{W}'$ and $(\varrho'',S^{\mathcal{W}'}, T^{(1)},\dots, T^{(k)})$ be a pure isometric representaion of $\mathbb{E}$ on $\mathcal{F}({E}_1)\ot \mathcal{W}'.$ By Lemma \ref{SS2}, for each $i \in I_k,$ there exist isometric bi-module maps $\Phi_{i}': E_{i+1} \to B(\mathcal{W}', \mathcal{F}(E_1) \otimes \mathcal{W}') $ such that $T^{(i)}=M_{\Phi_{i}'}.$ 
	In the setting of Theorem \ref{A6}, let $\Psi':\mathcal{W}'\to \mathcal{K}$ be an inner operator for the representations $(\varrho'',S^{\mathcal{W}'},M_{\Phi_{1}'},\dots,M_{\Phi_{k}'})$ and $(\rho,S,M_{\Theta_1},\dots,M_{\Theta_k})$ such that $\mathcal{M}=M_{\Psi'}(\mathcal{F}({E}_1)\ot \mathcal{W}').$  Define a bounded operator  $U:\mathcal{F}({E}_1)\ot \mathcal{W} \to \mathcal{F}({E}_1)\ot \mathcal{W}'$ by $$U=M_{\Psi'}^*M_{\Psi}.$$
	Since $M_{\Psi}(\mathcal{F}({E}_1)\ot \mathcal{W})=M_{\Psi'}(\mathcal{F}({E}_1)\ot \mathcal{W}'),$ $U$ is the unitary operator  that satisfies  $M_{\Psi}=M_{\Psi'}U.$ Moreover, $U$ intertwine the isometric covariant  representations   $(\varrho',S^{\mathcal{W}},M_{\Phi_1},\dots,M_{\Phi_k})$  and $(\varrho'',S^{\mathcal{W}'},M_{\Phi_{1}'},\dots,M_{\Phi_{k}'}).$  Indeed, from Equation (\ref{A4}) and  the definition  of $\Psi',$   \begin{align*}
		M_{\Psi'}U\widetilde{M}_{\Phi_{i}}&=M_{\Psi}\widetilde{M}_{\Phi_{i}}=\widetilde{M}_{\Theta_i}(I_{E_{i+1}}\ot M_{\Psi})\\&=\widetilde{M}_{\Theta_i}(I_{E_{i+1}}\ot M_{\Psi'}U)=M_{\Psi'}\widetilde{M}_{\Phi_{i}'}
		(I_{E_{i+1}}\ot U).
	\end{align*} But $M_{\Psi'}$ is an isometry, we have $U\widetilde{M}_{\Phi_{i}}=\widetilde{M}_{\Phi_{i}'}
	(I_{E_{i+1}}\ot U).$ Similarly, we prove  $U\varrho'(b)=\varrho''(b)U$ and $U\widetilde{S}^{\mathcal{W}}=\widetilde{S}^{\mathcal{W}'}(I_{E_{1}}\ot U).$
\end{remark}
Let $\mathcal{N}$ be a doubly commuting invariant subspace for a completely bounded covariant representation $(\sigma , V^{(1)},\dots,V^{(k)})$ of $\mathbb{E}$ on a Hilbert sapce $\mathcal{H},$ that is, $\mathcal{N}$ is   invariant  and  $(\sigma , V^{(1)},\dots,V^{(k)})|_{\mathcal{N}}$ is  doubly commuting.  In \cite[Theorem 4.11]{TV21} completely characterize the doubly commuting subspaces. However, we prove the same result in a different approach here: If $\mathcal{M} \subseteq \mathcal{F}({E}_1)\ot \mathcal{K}$   is a doubly commuting invariant subspace for $(\rho,S,M_{\Theta_1},\dots,M_{\Theta_k}),$  then   $(\varrho',S^{\mathcal{W}},M_{\Phi_1},\dots,M_{\Phi_k})$  be defined as in the proof of  Theorem \ref{A6} is  also  doubly commuting. In adddition to this and Theorem \ref{ZZZ2},  the 
following corollary    analyzes the doubly commuting invariant subspaces in the context of Theorem \ref{A6}.
\begin{corollary}
	Let $\mathcal{N}$ be an invariant subspace for a pure isometric covariant representation $(\sigma , V^{(1)},\dots,V^{(k+1)})$  of $\mathbb{E}$ on a Hilbert sapce $\mathcal{H}.$
	Then $\mathcal{N}$ is doubly commuting  for $(\sigma , V^{(1)},\dots,V^{(k+1)})$ if and only if there exist a representation of $\mathcal{B}$ acting on a Hilbert space $\mathcal{W},$ a doubly commuting pure isometric representation  $(\varrho',S^{\mathcal{W}},M_{\Phi_1},\dots,M_{\Phi_k})$ of ${\mathbb{E}}$ on $\mathcal{F}({E}_1)\ot \mathcal{W}$ and an inner operator  $\Psi :\mathcal{W}\to  \mathcal{H}$ $($for the representations $(\varrho',S^{\mathcal{W}},M_{\Phi_1},\dots,M_{\Phi_k})$ and $(\sigma , V^{(1)},\dots,V^{(k+1)})$ such that $$\mathcal{N}=M_{\Psi}(\mathcal{F}({E}_1)\ot \mathcal{W}).$$
\end{corollary}


Let $\mathcal{M}$ be an invariant subspace for $(\rho,S)$ of $E_1$ on $\mathcal{F}({E}_1)\ot \mathcal{K},$ from Theorem \ref{WWWW8} there exists an inner operator $\Psi :\mathcal{W}\to \mathcal{K}$ such that $\mathcal{M}=M_{\Psi}(\mathcal{F}({E}_1)\ot \mathcal{W}),$ where $\mathcal{W}=\mathcal{M}\ominus \widetilde{S}(E_1\ot \mathcal{M}).$ The corresponding question is that what additional conditions on $\Psi$ is $\mathcal{M}$ also a invariant subspace for $(\rho,M_{\Theta_1},\dots,M_{\Theta_k})?$ We get the answer of this question by reformulating Remark \ref{A92} and Theorem \ref{A6}.

\begin{theorem}\label{AAAA}
	Let $\mathcal{M}$ be a closed invariant subspace for $(\rho,S)$ of $E_1$ on $\mathcal{F}({E}_1)\ot \mathcal{K}$ and $\Psi:\mathcal{W}\to \mathcal{K}$  be an inner operator  for the representations $(\varrho',S^{\mathcal{W}})$ and $(\rho,S)$ such that $\mathcal{M}=M_{\Psi}(\mathcal{F}({E}_1)\ot \mathcal{W}),$  where $\mathcal{W}=\mathcal{M}\ominus \widetilde{S}(E_1\ot \mathcal{M}).$ Then $\mathcal{M}$ is invariant for $(\rho,M_{\Theta_1},\dots,M_{\Theta_k})$ if and only if there exist a pure isometric representation $(\varrho',M_{\Phi_1},\dots,M_{\Phi_k})$ of the product system  determined by $\{E_2,\dots, E_{k+1}\}$ on $\mathcal{F}({E}_1)\ot \mathcal{W}$ such that $$M_{\Psi}\varrho'(b)=\rho(b)M_{\Psi} \quad and \quad M_{\Psi}\widetilde{M}_{\Phi_i}=\widetilde{M}_{\Theta_i}(I_{E_{i+1}}\ot M_{\Psi}), \quad \quad i\in I_k,$$ where $\Phi_i(\xi_i)=\sum_{n \in \mathbb{N}_0}\widetilde{S}_{n}^{\mathcal{W}}(I_{E_1^{\otimes n}} \otimes P_{\mathcal{W}}M_{\Psi})\widetilde{S}_{n}^{\mathcal{W}^*}M_{\Psi}^*M_{\Theta_i}(\xi_i)|_{\mathcal{W}}, ~\xi_i \in E_{i+1}, i\in I_k.$ Moreover, the representations $(\rho,S,M_{\Theta_1},\dots,M_{\Theta_k})|_{\mathcal{M}}$ and $(\varrho',S^{\mathcal{W}} ,M_{\Phi_1},\dots,M_{\Phi_k})$ are isomorphic. 
\end{theorem}

\section{Applications}

In this section, we study some useful results about the structure of invariant subspaces and  the characterization of nested invariant subspaces    for doubly commuting  pure isometric representation of the product system over $\mathbb{N}_0^k$.  In addition, we also verify our important results for the invariant subspaces that are isomorphic to $\mathcal{F}({E}_1)\ot \mathcal{K}.$


First, we  discuss the characterization of nested invariant subspaces for pure doubly commuting isometric representation: 	Let $(\sigma , V^{(1)},\dots,V^{(k+1)})$ be a doubly commuting pure isometric representation of $\mathbb{E}$ acting on the Hilbert space $\mathcal{H}.$ By Theorem \ref{ZZZ2}, $(\sigma , V^{(1)},\dots,V^{(k+1)})$ is isomorphic to a covariant representation $(\rho,S,M_{\Theta_1},\dots,M_{\Theta_k})$ of ${\mathbb{E}}$ on $\mathcal{F}({{E}_{1}})\ot \mathcal{K},$ where  $(\rho,S)$ is the induced representaion  of $E_1$ on $\mathcal{F}({{E}_{1}})\ot \mathcal{K}.$ Therefore,  considering such a representation $(\rho,S,M_{\Theta_1},\dots,M_{\Theta_k}),$ it is sufficient to analyze the nested invariant subspaces.

Let   $\mathcal{M}_1$ and $\mathcal{M}_2$ be   invariant subspaces for $(\rho,S,M_{\Theta_1},\dots,M_{\Theta_k}),$   by Theorem \ref{A6}, there exist an inner operators ${\Psi}_1 : \mathcal{W}_1\to  \mathcal{K}$ (for the representations $(\varrho'_1,S^{\mathcal{W}_1},M_{\Phi_{1,1}},\dots,M_{\Phi_{k,1}})$ and $(\rho,S,M_{\Theta_1},\dots,M_{\Theta_k})$) and ${\Psi}_2 : \mathcal{W}_2\to  \mathcal{K}$ (for the representations $(\varrho'_2,S^{\mathcal{W}_2},M_{\Phi_{1,2}},\dots,M_{\Phi_{k,2}})$ and $(\rho,S,M_{\Theta_1},\dots,M_{\Theta_k})$) such that $$\mathcal{M}_1=M_{\Psi_1}(\mathcal{F}({E}_1)\ot \mathcal{W}_1) \quad and\quad \mathcal{M}_2=M_{\Psi_2}(\mathcal{F}({E}_1)\ot \mathcal{W}_2),$$  where $\Phi_{i,j}(\xi_i)=\sum_{n \in \mathbb{N}_0}\widetilde{S}_n^{\mathcal{W}_j}(I_{E_1^{\otimes n}} \otimes P_{\mathcal{W}_j}P_{\mathcal{M}_j})\widetilde{S}_{n}^{*}M_{\Theta_i}(\xi_i)|_{\mathcal{W}_j}, ~\xi_i \in E_{i+1}, i\in I_k, j\in I_2$ and $ \mathcal{W}_j={\mathcal{M}_j} \ominus \widetilde{S}(E_1 \ot {\mathcal{M}_j}).$ Define a bounded operator $\Pi: \mathcal{F}({E}_1)\ot  \mathcal{W}_1\to \mathcal{F}({E}_1)\ot \mathcal{W}_2$ by
\begin{align*}
	\Pi=M_{\Psi_2}^*M_{\Psi_1}.
\end{align*} From Equations (\ref{A3}) and (\ref{A4}) for $M_{\Psi_1}$ and $M_{\Psi_2},$ we get $$\Pi\varrho'_1(b)=\varrho'_2(b)\Pi, \quad 	\Pi\widetilde{S}^{\mathcal{W}_1}=\widetilde{S}^{\mathcal{W}_2}(I_{E_1}\ot \Pi)$$ and $$\Pi\widetilde{M}_{\Phi_{i,1}}=\widetilde{M}_{\Phi_{i,2}}(I_{E_{i+1}}\ot \Pi), \quad  \quad i\in I_k.$$ This shows that $\Pi$ is  multi-analytic   for the representations $(\varrho'_1,S^{\mathcal{W}_1},M_{\Phi_{1,1}},\dots,M_{\Phi_{k,1}})$ and $(\varrho'_2,S^{\mathcal{W}_2},M_{\Phi_{1,2}},\dots,M_{\Phi_{k,2}}).$
Furthermore, let's assume that   $\mathcal{M}_1\subseteq\mathcal{M}_2,$  then $\Pi$ is  an  isometry and $M_{\Psi_1}=M_{\Psi_2}\Pi,$  and therefore  there exists an inner operator $\Psi:\mathcal{W}_1\to  \mathcal{W}_2$ such that  $\Pi=M_{\Psi}.$ Note that $M_{\Psi_2}=\sum_{{n} \in \mathbb{N}_0} \widetilde{S}_n(I_{E_1^{\ot n}} \otimes \Psi_2)\widetilde{S}_n^{\mathcal{W}_2^*}$ and $\widetilde{S}_n^*|_{\mathcal{K}}=0$ for all $n\in \mathbb{N},$ we get $M_{\Psi_2}^*|_{\mathcal{K}}={\Psi_2}^*.$ Hence
$\Psi={M_{\Psi}}|_{\mathcal{W}_1}=M_{\Psi_2}^*M_{\Psi_1}|_{\mathcal{W}_1}={\Psi_2}^*{\Psi_1}.$  On the other hand, suppose  that $M_{\Psi_1}=M_{\Psi_2}M_{\Psi}$ and consider the subspaces  $\mathcal{M}_1=M_{\Psi_1}(\mathcal{F}({E}_1)\ot \mathcal{W}_1)$ and $\mathcal{M}_2= M_{\Psi_2}(\mathcal{F}({E}_1)\ot \mathcal{W}_2),$  thus it seems obviously that  $\mathcal{M}_1\subseteq \mathcal{M}_2.$ This observation leads to the following theorem which  characterize the nested invariant subspaces.


\begin{theorem}
	Let $\mathcal{M}_1$ and $\mathcal{M}_2$ be two invariant subspaces for $(\rho,S,M_{\Theta_1},\dots,M_{\Theta_k})$ of ${\mathbb{E}}$ on $\mathcal{F}({E}_1)\ot \mathcal{K}$ and $\Psi_j:\mathcal{W}_j\to \mathcal{K},j\in I_2$ be an inner operators (for the representations $(\varrho'_j,S^{\mathcal{W}_j},M_{\Phi_{1,j}},\dots,M_{\Phi_{k,j}})$ and $(\rho,S,M_{\Theta_1},\dots,M_{\Theta_k})$) such that $\mathcal{M}_j=M_{\Psi}(\mathcal{F}({E}_1)\ot \mathcal{W}_j),$  where $ \mathcal{W}_j={\mathcal{M}_j} \ominus \widetilde{S}(E_1 \ot {\mathcal{M}_j}).$ Then $\mathcal{M}_1\subseteq\mathcal{M}_2$ if and only if there exists an inner operator ${\Psi} : \mathcal{W}_1\to  \mathcal{W}_2$ (for the representations $(\varrho'_1,S^{\mathcal{W}_1},M_{\Phi_{1,1}},\dots,M_{\Phi_{k,1}})$ and $(\varrho'_2,S^{\mathcal{W}_2},M_{\Phi_{1,2}},\dots,M_{\Phi_{k,2}})$) such that $$M_{\Psi_1}=M_{\Psi_2}M_{\Psi},$$ where $\Phi_{i,j}(\xi_i)=\sum_{n \in \mathbb{N}_0}\widetilde{S}_n^{\mathcal{W}_j}(I_{E_1^{\otimes n}} \otimes P_{\mathcal{W}_j}P_{\mathcal{M}_j})\widetilde{S}_{n}^{*}M_{\Theta_i}(\xi_i)|_{\mathcal{W}_j}, ~\xi_i \in E_{i+1}, i\in I_k,j\in I_2.$
\end{theorem}


Now we  discuss the structure of the invariant subspaces for pure doubly commuting isometric representation. We have sufficient data from Theorem \ref{ZZZ2} to consider the representation $(\rho,S,M_{\Theta_1},\dots,M_{\Theta_k})$ on $\mathcal{F}({{E}_{1}})\ot \mathcal{K},$ and discuss about the invariant subspces. 
First, we start with the following definition:
Let $(\sigma,V^{(1)},\dots, V^{(k)})$ and $(\psi, T^{(1)},\dots, T^{(k)} )$ be  c.b.c. representations of $\mathbb{E}$ acting on the Hilbert spaces $\mathcal{H}$ and $\mathcal{H}',$ respectively. Let $\mathcal{M}\subseteq \mathcal{H}$ and $\mathcal{M}'\subseteq \mathcal{H}'$ be  invariant subspaces for the representations $(\sigma,V^{(1)},\dots, V^{(k)})$ and $(\psi, T^{(1)},\dots, T^{(k)}),$ respectively. We say that the subspaces $\mathcal{M}$ and $\mathcal{M}'$ are {\it isomorphic} if $(\sigma,V^{(1)},\dots, V^{(k)})|_{\mathcal{M}}$ is isomorphic to $(\psi, T^{(1)},\dots, T^{(k)} )|_{\mathcal{M}'}.$

Suppose that  $(\sigma , V^{(1)},\dots,V^{(k)})$  and   $(\psi, T^{(1)},\dots, T^{(k)} )$ are    doubly commuting pure isometric representations.   From Theorem \ref{ZZZ2} $(k=1$ case$),$ there exist a Hilbert space $\mathcal{K}$ and a unitary $U_{\mathcal{H}}:\mathcal{H}\to \mathcal{F}({E}_1)\ot \mathcal{K}$ such that $$U_{\mathcal{H}}\sigma(b)=\rho(b) U_{\mathcal{H}},\quad  U_{\mathcal{H}}\wV^{(1)}=\widetilde{S}(I_{E_1}\ot U_{\mathcal{H}})$$  and  $$U_{\mathcal{H}}\wV^{(i+1)}=\widetilde{M}_{\Theta_i}(I_{E_{i+1}}\ot U_{\mathcal{H}}),\quad  i\in I_k.$$
Similarly, for  the representaion $(\psi, T^{(1)},\dots, T^{(k+1)})$ of $\mathbb{E}$ on  $\mathcal{H}',$ there exist a Hilbert space $\mathcal{K}'$ and a unitary $U_{\mathcal{H}'}:\mathcal{H}'\to \mathcal{F}({E}_1)\ot \mathcal{K}'$ such that $$U_{\mathcal{H}'}\psi(b)=\rho'(b) U_{\mathcal{H}'},\quad U_{\mathcal{H}'}\wT^{(1)}=\widetilde{S}'(I_{E_1}\ot U_{\mathcal{H}'})$$ and   $$U_{\mathcal{H}'}\wT^{(i+1)}=\widetilde{M}_{\Theta_i'}(I_{E_{i+1}}\ot U_{\mathcal{H}'}), \quad  i\in I_k.$$  In fact, $U_{\mathcal{H}}$ and $U_{\mathcal{H}'}$ are mutli-analytic.
Suppose that $A:\mathcal{H}\to \mathcal{H}'$ is a multi-analytic operator  for the representaions  $(\sigma , V^{(1)},\dots,V^{(k+1)})$  and   $(\psi, T^{(1)},\dots, T^{(k)} ),$ then the operator  $A':\mathcal{F}({E}_1)\ot \mathcal{K}\to \mathcal{F}({E}_1)\ot \mathcal{K}'$ define  by $A'=U_{\mathcal{H}'}AU_{\mathcal{H}}^*$   is also multi-analytic for the representaions $(\rho,S,M_{\Theta_1},\dots,M_{\Theta_k})$ and $(\rho',S',M_{\Theta_1'},\dots,M_{\Theta_k'}).$ 

On the other hand, suppose $A':\mathcal{F}({E}_1)\ot \mathcal{K}\to \mathcal{F}({E}_1)\ot \mathcal{K}'$ is  multi-analytic for  the representaions $(\rho,S,M_{\Theta_1},\dots,M_{\Theta_k})$ and $(\rho',S',M_{\Theta_1'},\dots,M_{\Theta_k'}).$ Then the operator $A$ yields a multi-analytic for the representaions $(\sigma , V^{(1)},\dots,V^{(k+1)})$  and   $(\psi, T^{(1)},\dots, T^{(k)} ),$ namely $A=U_{\mathcal{H}'}^*A'U_{\mathcal{H}}.$ Moreover, the  construction above  shows that  $A \in B(\mathcal{H},\mathcal{H}')$ is a contraction (resp. isometry, unitary) if and only if  $A' \in B(\mathcal{F}({E}_1)\ot \mathcal{K}, \mathcal{F}({E}_1)\ot \mathcal{K}')$  is a contraction (resp. isometry, unitary).





\begin{definition}
	Let $\pi_1$ and $\pi_2$ be  representations of $\mathcal{B}$ acting on the Hilbert spaces $\mathcal{W}$ and $\mathcal{W}',$ respectively. For $i\in I_k,$ let $\Phi_i:E_i\to B(\mathcal{W},\mathcal{F}({E}_1)\ot \mathcal{W})$ and $\Phi'_i:E_i\to B(\mathcal{W}',\mathcal{F}({E}_1)\ot \mathcal{W}')$ be  completely bounded bi-module maps such that $(\rho,M_{\Phi_1}, \dots,M_{\Phi_k})$ and $(\rho',M_{\Phi_1'},\dots,M_{\Phi_k'})$  are c.b.c representations of $\mathbb{E}$ on  $\mathcal{F}({E}_1)\ot \mathcal{W}$  and $\mathcal{F}({E}_1)\ot \mathcal{W}',$ respectively.   We say that the pairs $(\rho,{\Phi}_1,\dots,{\Phi}_k)$ and $(\rho',{\Phi}'_1,\dots,{\Phi}'_k)$ {\rm coincide} if there exists a unitary $Z:\mathcal{W}\to \mathcal{W}'$ such that $$(I_{\mathcal{F}({E}_1)}\ot Z)\rho(b)=\rho'(b)(I_{\mathcal{F}({E}_1)}\ot Z)$$ and $$ (I_{\mathcal{F}({E}_1)}\ot Z)\widetilde{\Phi}_i=\widetilde{\Phi}_i'(I_{E_{i}}\ot Z), \quad i \in I_k.$$
\end{definition}

The following theorem is a complete set of isomorphic invariants of the corresponding invariant subspaces.

\begin{theorem}\label{CCC}
	Let $(\sigma,V^{(1)},\dots, V^{(k+1)})$ and $(\psi, T^{(1)},\dots, T^{(k+1)} )$ be doubly commuting pure isometric representations of $\mathbb{E}$ acting on the Hilbert spaces $\mathcal{H}$ and $\mathcal{H}',$ respectively. Let $\mathcal{M}\subseteq \mathcal{F}({E}_1)\ot \mathcal{K}$ and $\mathcal{M}'\subseteq \mathcal{F}({E}_1)\ot \mathcal{K}'$ be  invariant subspaces for the representations $(\rho,S,M_{\Theta_1},\dots,M_{\Theta_k})$ and $(\rho',S',M_{\Theta_1'},\dots,M_{\Theta_k'})$ of ${\mathbb{E}}$ acting on the Hilbert spaces $\mathcal{F}({E}_1)\ot \mathcal{K}$ and $\mathcal{F}({E}_1)\ot \mathcal{K}',$ respectively. Then $\mathcal{M}$ is isomorphic to $\mathcal{M}'$ if and only if $\{\varrho',{\Phi}_1,\dots,{\Phi}_k\}$ is coincide with $\{\varrho'',{\Phi}'_1,\dots,{\Phi}'_k\}.$
\end{theorem}
\begin{proof}
	Let $\mathcal{M}$ be a inavriant subspace for  $(\rho,S,M_{\Theta_1},\dots,M_{\Theta_k}),$ by Theorem \ref{A6} and Remark \ref{A93},  there exist a representaion of $\mathcal{B}$ acting on a Hilbert space $\mathcal{W},$ completely bi-module maps $\Phi_i:E_{i+1}\to B(\mathcal{W},\mathcal{F}({E}_1)\ot \mathcal{W})$ such that    $(\varrho',S^{\mathcal{W}} ,M_{\Phi_1},\dots,M_{\Phi_k})$  is isomorpic to $(\rho,S,M_{\Theta_1},\dots,M_{\Theta_k})|_{\mathcal{M}}.$ Moreover, $\mathcal{W}=\mathcal{M}\ominus \widetilde{S}(E_1\ot \mathcal{M})$  and  $\Phi_i(\xi_i)=\sum_{n \in \mathbb{N}_0}\widetilde{S}_n^{\mathcal{W}}(I_{E_1^{\otimes n}} \otimes P_{\mathcal{W}}P_{\mathcal{M}})\widetilde{S}_{n}^{*}M_{\Theta_i}(\xi_i)|_{\mathcal{W}}$ $\xi_i \in E_{i+1}, i\in I_k.$
	Suppose $\mathcal{M}\subseteq \mathcal{F}({E}_1)\ot \mathcal{K}$ is isomorphic to $\mathcal{M}'\subseteq \mathcal{F}({E}_1)\ot \mathcal{K}',$ i.e., the representations $(\rho,S,M_{\Theta_1},\dots,M_{\Theta_k})|_{\mathcal{M}}$   and $(\rho',S',M_{\Theta_1'},\dots,M_{\Theta_k'})|_{\mathcal{M}'}$  are isomorphic, then using the corresponding covariant representations     $(\varrho',S^{\mathcal{W}} ,M_{\Phi_1},\dots,M_{\Phi_k})$  of $(\rho,S,M_{\Theta_1},\dots,M_{\Theta_k})|_{\mathcal{M}}$ and $(\varrho'',S^{\mathcal{W}'} ,M_{\Phi'_1},\dots,M_{\Phi'_k})$  of $(\rho',S',M_{\Theta_1'},\dots,M_{\Theta_k}')|_{\mathcal{M}'}$ from above are isomorphic. 
	This implies that there exists a unitary operator $U:\mathcal{F}({E}_1)\ot \mathcal{W} \to \mathcal{F}({E}_1)\ot \mathcal{W}'$ such that
	\begin{align*}
		U\varrho'(b)=\varrho''(b)U,~  U\widetilde{S}^{\mathcal{W}}=\widetilde{S}^{\mathcal{W}'}(I_{E_1}\ot U), \quad U\widetilde{M}_{\Phi_{i}}=\widetilde{M}_{\Phi'_{i}}(I_{E_{i+1}}\ot U),
	\end{align*} for all $i\in I_k.$
	Note that  the unitary operator  $U$ is multi-analytic, thus $U=M_{\theta}$ for some inner operator $\theta: \mathcal{W} \rightarrow \mathcal{W}'.$ In fact,  the corresponding inner operator $\theta$ is onto. Since $U$ intertwine the representations $(\varrho',S^{\mathcal{W}})$ and $(\varrho'',S^{\mathcal{W}'})$ and by the definition  of $M_{\theta}$ (see Equation (\ref{SS1})),	we can easily see that   $U=I_{\mathcal{F}({E}_1)}\ot Z,$ where $Z=\theta.$ Using the above  equation, we get $$  (I_{\mathcal{F}({E}_1)}\ot Z)\widetilde{\Phi}_i=\widetilde{\Phi}_i'(I_{E_{i+1}}\ot Z), \quad i \in I_k.$$ Conversely, suppose that $\{\varrho',{\Phi}_1,\dots,{\Phi}_k\}$ and  $\{\varrho'',{\Phi}'_1,\dots,{\Phi}'_k\}$ are coincide, there exists a unitary $Z:\mathcal{W}\to \mathcal{W}'$ such that $$(I_{\mathcal{F}({E}_1)}\ot Z)\varrho'(b)=\varrho''(b)(I_{\mathcal{F}({E}_1)}\ot Z) \quad and\quad (I_{\mathcal{F}({E}_1)}\ot Z)\widetilde{\Phi}_i=\widetilde{\Phi}_i'(I_{E_{i+1}}\ot Z),$$ for all $i\in I_k.$ Clearly $U=I_{\mathcal{F}({E}_1)}\ot Z$ is a unitary from $\mathcal{F}({E}_1)\ot \mathcal{W}$ to $\mathcal{F}({E}_1)\ot \mathcal{W}'.$ From Equation (\ref{SS1}), it is easy to verify that $	U\widetilde{S}^{\mathcal{W}}=\widetilde{S}^{\mathcal{W}'}(I_{E_1}\ot U)$ and $ U\widetilde{M}_{\Phi_{i}}=\widetilde{M}_{\Phi'_{i}}(I_{E_{i+1}}\ot U)$ for $i\in I_k.$ Then, by Remark \ref{A93}, $\mathcal{M}$ is isomorphic to $\mathcal{M}'.$
\end{proof}

Form the above theorem, since $\mathcal{M}$ and $\mathcal{M}'$ are  the invariant subspaces for  the representations $(\rho,S)$ and $(\rho',S')$ and by Theorem \ref{WWWW8}, there are Hilbert spaces $\mathcal{W}$ and $\mathcal{W}'$ such that  $$\mathcal{M}=M_{\Psi}(\mathcal{F}({E}_1)\ot \mathcal{W})\quad and \quad \mathcal{M}'=M_{\Psi'}(\mathcal{F}({E}_1)\ot \mathcal{W}'),$$ for some inner operators  $\Psi:\mathcal{W}\to \mathcal{K}$ and $\Psi':\mathcal{W}'\to \mathcal{K}'.$  Then by Remark \ref{A92}, the completely bounded bi-module maps $\Phi_i, \Phi'_i,\: i \in I_k,$  in Theorem \ref{CCC} can be written as 
$$\Phi_i(\xi_i)=\sum_{n \in \mathbb{N}_0}\widetilde{S}_{n}^{\mathcal{W}}(I_{E_1^{\otimes n}} \otimes P_{\mathcal{W}}M_{\Psi})\widetilde{S}_{n}^{\mathcal{W}^*}M_{\Psi}^*M_{\Theta_i}(\xi_i)|_{\mathcal{W}}$$ and $$\Phi'_i(\xi_i)=\sum_{n \in \mathbb{N}_0}\widetilde{S}_{n}^{\mathcal{W}'}(I_{E_1^{\otimes n}} \otimes P_{\mathcal{W}'}M_{\Psi'})\widetilde{S}_{n}^{\mathcal{W}'^*}M_{\Psi'}^*M_{\Theta'_i}(\xi_i)|_{\mathcal{W}'}$$ for all $\xi_i \in E_{i+1}, i\in I_k .$

The  following theorem characterize the  invariant subspaces for  $(\rho,S,M_{\Theta_1},\dots,M_{\Theta_k})$ which  is isomorphic to $\mathcal{F}({E}_1)\ot \mathcal{K}'.$
\begin{theorem}\label{LLL}
	Let $(\sigma,V^{(1)},\dots, V^{(k+1)})$ and $(\psi, T^{(1)},\dots, T^{(k+1)} )$ be two doubly commuting pure isometric representations of $\mathbb{E}$ acting on the Hilbert spaces $\mathcal{H}$ and $\mathcal{H}',$ respectively. Let $\mathcal{M}\subseteq \mathcal{F}({E}_1)\ot \mathcal{K}$ be an invariant subspace for $(\rho,S,M_{\Theta_1},\dots,M_{\Theta_k})$ of ${\mathbb{E}}$ on $\mathcal{F}({E}_1)\ot \mathcal{K}$. Then $\mathcal{M}$ is isomorphic to $\mathcal{F}({E}_1)\ot \mathcal{K}'$ if and only if there exists an inner operator $\Psi :\mathcal{K}'\to  \mathcal{K}$ ( for the representations $(\rho',S',M_{\Theta_1'},\dots,M_{\Theta_k'})$ and $(\rho,S,M_{\Theta_1},\dots,M_{\Theta_k})$) such that $$\mathcal{M}=M_{\Psi}(\mathcal{F}({E}_1)\ot \mathcal{K}').$$
\end{theorem}
\begin{proof}
	Suppose $\mathcal{M}$ is isomorphic to $\mathcal{F}({E}_1)\ot \mathcal{K}',$ i.e., $(\rho,S,M_{\Theta_1},\dots,M_{\Theta_k})|_{\mathcal{M}}$ is isomorphic to $(\rho',S',M_{\Theta_1'},\dots,M_{\Theta_k'}),$ then there exists a unitary $U:\mathcal{F}({E}_1)\ot \mathcal{K}'\to \mathcal{M}$ such that $$U\rho'(b)=\rho(b)U, ~ U \widetilde{S}'=\widetilde{S}(I_{E_1}\ot U)\quad and \quad U \widetilde{M}_{\Theta_i'}=\widetilde{M}_{\Theta_i}(I_{E_{i+1}}\ot U).$$
	Let $i_{\mathcal{M}}$ be the inclusion map from ${\mathcal{M}}$ to $\mathcal{F}({E}_1)\ot \mathcal{K}.$ Define an isometry $\Pi: \mathcal{F}(E_1)\ot \mathcal{K}'\to \mathcal{F}({E}_1)\ot \mathcal{K}$ by $\Pi=i_{\mathcal{M}}U,$ then $\Pi \Pi^*=i_{\mathcal{M}}i_{\mathcal{M}}^*,$ $$\Pi \rho'(b)=\rho(b)\Pi \quad \Pi\widetilde{S}'=\widetilde{S}(I_{E_1}\ot \Pi), \quad and \quad \Pi\widetilde{M}_{\Theta_i'}=\widetilde{M}_{\Theta_i}(I_{E_{i+1}}\ot \Pi).$$ Therefore  $\Pi$ is an isometric multi-analytic operator and thus  there exists an inner operator $\Psi:\mathcal{K}'\to  \mathcal{K}$ such that  $\Pi=M_{\Psi}.$ This shows that $$\mathcal{M}=\Pi(\mathcal{F}(E_1)\ot \mathcal{K}')=M_{\Psi}(\mathcal{F}(E_1)\ot \mathcal{K}').$$ 
	
	On the other hand, let $\Psi:\mathcal{K}'\rightarrow \mathcal{K}$ be an inner  operator such that $\mathcal{M}=M_{\Psi}(\mathcal{F}({E}_1)\ot \mathcal{K}')\subseteq \mathcal{F}({E}_1)\ot \mathcal{K},$ then $M_{\Psi}:\mathcal{F}({E}_1)\ot \mathcal{K}'\to \mathcal{M}$ is a unitary which intertwines  the representations $(\rho',S',M_{\Theta_1'},\dots,M_{\Theta_k'})$  and $(\rho,S,M_{\Theta_1},\dots,M_{\Theta_k})|_{\mathcal{M}}.$ Then $\mathcal{M}$ is isomorphic to $\mathcal{F}({E}_1)\ot \mathcal{K}'.$
\end{proof}

\begin{remark}
	Let $\mathcal{M}\subseteq \mathcal{F}({E}_1)\ot \mathcal{K}$ be a doubly commuting  subspace for $(\rho,S,M_{\Theta_1},\dots,M_{\Theta_k})$  on $\mathcal{F}({E}_1)\ot \mathcal{K},$ then as an invariant subspace $\mathcal{M},$ $(\varrho',S^{\mathcal{W}},M_{\Phi_1},\dots,M_{\Phi_k})$ on $\mathcal{F}({E}_1)\ot \mathcal{W}$   be defined as in the proof of  Theorem \ref{A6} is  isomorphic to $(\rho,S,M_{\Theta_1},\dots,M_{\Theta_k})|_{\mathcal{M}},$ that is,  $\mathcal{M}$ is isomorphic to $\mathcal{F}({E}_1)\ot \mathcal{W}.$ Therefore, by Theorem \ref{LLL}, $\mathcal{M}=M_{\Psi}(\mathcal{F}({E}_1)\ot \mathcal{W}),$ for some inner operator $\Psi :\mathcal{W}\to  \mathcal{K}.$ Since $\mathcal{M}$ is the doubly commuting subspace,  $(\varrho',S^{\mathcal{W}},M_{\Phi_1},\dots,M_{\Phi_k})$ on $\mathcal{F}({E}_1)\ot \mathcal{W}$ is pure doubly commuting isometric representation. 
\end{remark}

\subsection*{Acknowledgment}
Dimple Saini is supported by UGC fellowship (File No:16-6(DEC. 2018)/\\2019(NET/CSIR)). Harsh Trivedi is supported by MATRICS-SERB  
Research Grant, File No: MTR/2021/000286, by 
SERB, Department of Science \& Technology (DST), Government of India. Shankar Veerabathiran thanks IMSc Chennai for postdoc fellowship. Saini and Trivedi acknowledge the DST-FIST program (Govt. of India) FIST - No. SR/FST/MS-I/2018/24.

\end{document}